\documentclass[12pt]{article}
\usepackage[dvips]{graphicx}
\usepackage{amsmath,amsfonts,amssymb,amsthm,color}

\def\hsymb#1{\mbox{\strut\rlap{\smash{\huge $#1$}}\quad}}
\usepackage{natbib}

\newtheorem{theorem}{Theorem}
\newtheorem{lemma}{Lemma}
\newtheorem{corollary}{Corollary}

\newtheorem{proposition}{Proposition}

\newtheorem{definition}{Definition}

\newtheorem{conjecture}{Conjecture}

\usepackage{bm}
\bmdefine{\Bt}{t}
\bmdefine{\BX}{X}
\bmdefine{\BY}{Y}
\bmdefine{\BZ}{Z}
\bmdefine{\BB}{B}
\bmdefine{\BM}{M}
\bmdefine{\BD}{D}
\bmdefine{\Bi}{i}
\bmdefine{\Bj}{j}
\bmdefine{\Bk}{k}
\bmdefine{\Bx}{x}
\bmdefine{\By}{y}
\bmdefine{\Bz}{z}
\bmdefine{\Bv}{v}
\bmdefine{\Bw}{w}
\bmdefine{\Bn}{n}
\bmdefine{\Ba}{a}
\bmdefine{\Bb}{b}
\bmdefine{\Bc}{c}
\bmdefine{\Be}{e}
\bmdefine{\Bu}{u}
\bmdefine{\Bp}{p}
\bmdefine{\Bzero}{0}
\bmdefine{\Bone}{1}

\title{On connectivity of fibers with  positive marginals in multiple
  logistic regression}
\author{
Hisayuki Hara\footnote{
Department of Technology Management for Innovation, 
University of Tokyo}, 
Akimichi Takemura\footnote{
Graduate School of Information Science and Technology, 
University of Tokyo}
\footnote{CREST, JST} \ and 
Ruriko Yoshida\footnote{
Department of Statistics, 
University of Kentucky}
}
\date{October 2008}

\begin{document}
\maketitle

\begin{abstract}
%We study Markov bases for conducting 
In this paper we consider exact tests of a multiple logistic
regression, where the levels of covariates are equally spaced, via Markov 
beses. In
usual application of multiple logistic regression, the sample size is
positive for each combination of levels of the covariates.
In this case we do not need a whole Markov basis, which
guarantees connectivity of all fibers.
We first give an explicit Markov basis for
multiple Poisson regression.  
By the Lawrence lifting of this basis, in the case of bivariate
logistic regression, we show a simple subset 
of the Markov basis which connects all 
fibers with a positive sample size for each combination of
levels of covariates.
\end{abstract}


\section{Introduction}

\cite{diaconis-sturmfels} developed an algorithm for sampling from
conditional distributions for a statistical model of discrete exponential
families, based on the algebraic theory of toric ideals.
This algorithm is applied to categorical data analysis through the notion of
Markov bases. 
%Markov bases introduced by \cite{diaconis-sturmfels} provide an
%attractive general
%methodology for conducting exact conditional tests of hypotheses
%linear in natural parameters in discrete exponential families.
However, often Markov bases are large and difficult to
compute.
One reason for their large size is that they
guarantee connectivity of all fibers (conditional sample spaces).  
%From the
%viewpoint of application, 
With a given particular data set, on the other hand, 
% side, 
we are
naturally interested in the connectivity of a particular fiber.
However obtaining a subset of a Markov basis for connecting a particular
fiber is also a difficult problem in general \citep{Ian2008}.  This problem was already
discussed in Section 3 of \cite{diaconis-sturmfels} concerning
``corner minors''.
In \cite{aoki-takemura-2005jscs} the case of two-way incomplete tables
was studied.

In most applications of the logistic regression model, for each
combination of covariates, the number of successes and the number of
failures are observed.  The number of trials (i.e.\ the sum of numbers of
successes and failures) for each combination of covariates is usually
fixed by a sampling scheme and positive.  We call this marginal {\em the
response variable marginal}.  Therefore we are usually interested in the
connectivity of fibers with positive response variable marginals
rather than all the fibers.  
First, in this paper, we show an explicit form of a %the minimum fiber 
Markov basis for multiple Poisson regression.  Then, extending the
result of \cite{chen-dinwoodie-dobra-huber2005}, 
we show an explicit form of a subset of Markov basis, which guarantees the
connectivity of every fiber with positive response variable marginals
for bivariate logistic regression.  We conjecture that
a similar subset of Markov basis connects fibers with positive response variable
marginals for a general multiple logistic regression. However, it seems difficult to prove this conjecture.  
% However our proof
%for the bivariate case is already very complicated.
%On the other hand we show an explicit form of the minimum fiber Markov
%basis for multiple Poisson regression. From the viewpoint of Markov
%bases, multiple Poisson regression is much easier than multiple
%logistic regression.

The logistic regression can
be understood as the Lawrence lifting  of a Poisson regression.  
%% Therefore we are
%% interested in the fibers of the Lawrence lifting with positive
%% marginal frequency for each combination of covariates.
Let $A$ denote a configuration defining a toric ideal and let
$\Lambda(A)$ denote its Lawrence lifting.  Let $I_A$ and
$I_{\Lambda(A)}$ denote the respective toric ideals. It is known
\citep[Theorem 7.1]{sturmfels1996} that the unique minimal Markov
basis of $I_{\Lambda(A)}$ coincides with the Graver basis of $I_A$.
Therefore the whole Graver basis of $I_A$ is needed to guarantee the
connectivity of all fibers of $\Lambda(A)$.  However many of the
elements of the Graver basis of $I_A$ seem to be needed to cope with
the case of zero response variable marginal frequencies.
In Section \ref{sec:connectivity}, for the case of bivariate logistic
regression, we prove that a smaller Markov basis for the Poisson regression 
extended to the logistic regression guarantees the connectivity of
fibers with positive response variable marginals. 

This paper is organized as follows. In Section \ref{sec:preliminaries}
%% we briefly summarize notations and definitions concerning toric ideals
%% and the present the minimum-fiber Markov basis for univariate Poisson
%% regression model.  
we summarize results on Markov basis of univariate Poisson regression
and results on the connectivity for fibers with positive response variable
marginals of univariate logistic  regression.
In Section \ref{sec:Segre} we prove a theorem on
Markov bases of Segre product of configurations and apply it to
multiple Poisson regression.  In Section \ref{sec:connectivity} we
prove the connectivity of fibers with positive response variable 
marginals in the case of
bivariate logistic regression.  Some numerical examples are given in
Section \ref{sec:examples}.  We conclude this paper with some
discussions in Section \ref{sec:discussions}.  Some detailed  proofs are
in Appendix.

\section{Univariate Poisson and logistic regressions}
%Preliminaries}
\label{sec:preliminaries}
In this section we  summarize results on Markov basis of univariate Poisson regression
and the connectivity results for fibers with positive response variable
marginals of univariate logistic regression. We provide exact
statements  and detailed proofs of these results, 
because they are not explicitly given in literature and similar arguments
will be repeatedly applied to prove our main theorem in Section 
\ref{sec:connectivity}.

Consider univariate Poisson regression \citep{diaconis-eisenbud-sturmfels1998}
with the set of levels $\{1,\dots, J\}$ of a covariate.
The mean $\mu_j$ of  independent Poisson random variables $X_j$,
$j=1,\dots,J$, % with 
is modeled as
\[
%\log \mu_i = \alpha + \beta j,  \quad j=1,\dots,J.
\log \mu_j = \alpha + \beta j,  \quad j=1,\dots,J.
\]
The sufficient statistic for the models is $(\sum_{j=1}^J X_j, \sum_{j=1}^J j X_j)$.
The first component is the total sample size $n=\sum_{j=1}^J X_j$.
%% In the literature on Markov bases and toric ideals, 
%% the matrix giving the relation between the observation 
%% vector and the sufficient statistic is called a configuration.
The configuration $A$, i.e.\ the matrix giving the relation between 
the observation  vector and the sufficient statistic,
for this model is given by 
\begin{equation}
\label{eq:configuratio-univariate-Poisson}
A = \begin{pmatrix} 1& 1 & \dots & 1 \\
                    1 & 2 & \dots & J \\
                  \end{pmatrix}.
\end{equation}
Now we show the minimum-fiber Markov basis
\citep{takemura-aoki-2005bernoulli}
for the   univariate Poisson regression.
The minimum-fiber Markov basis is the union of all minimal Markov bases.

\begin{proposition} 
 \label{prop:1}
 Let ${\Be_j}$ denote the contingency table with
 just $1$ frequency in the $j$-th cell.  The set of moves
 \begin{equation}
  \label{eq:zijji}
   {\cal B} =\{ \pm (\Be_{j_1}+\Be_{j_4} - \Be_{j_2}
   - \Be_{j_3}) \mid 1\le j_1 < j_2 \le j_3 < j_4 \le J, \ \ 
j_2-j_1=j_4-j_3\}
 \end{equation}
forms the minimum-fiber Markov basis for the
univariate Poisson regression.% with $J$ levels.
\end{proposition}

\begin{proof} We employ the distance reducing argument of 
 \cite{takemura-aoki-2005bernoulli}.
  Let $x=(x_1, \dots, x_J)$ and $y=(y_1, \dots, y_J)$, $x\neq y$, be
  two data sets in the same fiber.  Let $j_1 = \min\{ k \mid x_k \neq
  y_k \}$.  Consider the case $x_{j_1} > y_{j_1}$. The case $x_{j_1} < y_{j_1}$ can be
  handled by interchanging the roles of $x$ and $y$. Because the total
  sample size $n$ is the same in $x$ and $y$ (i.e.\ $n=\sum_{k=1}^J x_k =
  \sum_{k=1}^J y_k$), there exists some $j_2 > j_1$ such that
  $x_{j_2} < y_{j_2}$.  Choose the smallest such $j_2$.  Now suppose that 
  $x_k \le y_k$ for all $k\ge j_2$.  Then
\begin{align*}
 0 &= \sum_{k=1}^J k (y_k - x_k) \ge \sum_{k=j_2}^J j_2(y_k - x_k)
         - \sum_{k=1}^{j_2-1} (j_2-1)(x_k - y_k)\\
   &= j_2 \sum_{k=1}^J (y_k - x_k) + \sum_{k=1}^{j_2-1} (x_k - y_k)=
\sum_{k=1}^{j_2-1} (x_k - y_k) > 0, 
 \end{align*}
 which is a contradiction. Therefore there exists some $j_4>j_2$, such that
 $x_{j_4} > y_{j_4}$.  Define $j_3:=j_1+j_4-j_2$. Then
 $\Be_{j_1}+\Be_{j_4} - \Be_{j_2} - \Be_{j_3}$ can be subtracted from $x$ and the
 $L_1$ distance 
 to $y$ becomes smaller.  This proves that $\cal B$
 forms a Markov basis.

 Now consider a fiber ${\cal F}_{2,c}$ 
 with sample size $n=2$ and a particular value of
 $c=\sum_{k=1}^k k x_k$.   This fiber is written as
 \[
 {\cal F}_{2,c}= \{ \Be_j+\Be_j' \mid 1\le j \le j'\le J, \ j+j'=c \}
 \]
 and $\cal B$ consists of all the differences of two elements of
 these fibers.
 Since $\cal B$ forms a Markov basis,
 every minimal Markov basis needs to connect only these fibers.
 This proves that $\cal B$ is the minimum-fiber Markov basis
\end{proof}

We now consider univariate logistic regression
\citep{chen-dinwoodie-dobra-huber2005}.
%% we first give a brief review of Markov bases of univariate logistic
%% regression model. 
Let $\{1,\dots, J\}$ be the set levels of a covariate and
let  $X_{1j}$ and $X_{2j}$, $j=1,\ldots,J$, be the numbers of successes and
failures, respectively. The probability for success
$p_j$ is modeled as
\[
 \mathrm{logit}(p_j) = 
 \log
 \frac{p_j}{1-p_j}
 =
 \alpha + \beta j , \qquad j=1,\dots,J.
 \]
The sufficient statistics for the model is 
$(X_{1+}, X_{+1},\dots,X_{+J},  \sum_{j=1}^J j X_{+j})$. 
Hence moves $z=(z_{ij})$ for the model satisfy 
$(z_{1+}, z_{+1},\dots,z_{+J}) = 0$ and 
\begin{equation}
 \label{linear-constraint-1}
 \sum_{j=1}^J j z_{+j} = 0.
\end{equation}
%% Let $A$ be 
%% \[
%% A = \begin{pmatrix} 1& 1 & \dots & 1 \\
%%                     1 & 2 & \dots & J \\
%%                   \end{pmatrix}.
%% \]
The configuration for this model is the Lawrence lifting 
$\Lambda(A)$ of $A$ in  (\ref{eq:configuratio-univariate-Poisson}): 
\begin{equation}
\label{def:A}
 \Lambda(A)= 
 \begin{pmatrix}
  A & 0 \\
  E_J & E_J
 \end{pmatrix}, 
 \qquad
 A = \begin{pmatrix} 
      1& 1 & \dots & 1 \\
      1 & 2 & \dots & J \\
      \end{pmatrix},
\end{equation}
where $E_J$ denotes the $J\times J$ identity matrix.

In general Markov bases of $\Lambda(A)$ become very complicated.
In usual applications of the logistic regression model, however, 
$X_{+j} := X_{1j} + X_{2j}$ is fixed by a sampling scheme 
and positive.
\cite{chen-dinwoodie-dobra-huber2005} showed that a simple subset of
Markov bases of $\Lambda(A)$ guarantees the connectivity of all fibers
satisfying 
$(X_{+1},\dots,X_{+J}) > 0$, where 
the inequality ``$>0$'' means that 
every element is positive.

Let $\bm{e}_j$ be redefined by a $2 \times J$ integer array with 
1 in the $(1,j)$-cell and $-1$ in the $(2,j)$-cell.
Then we can show that the  set of moves in (\ref{eq:zijji})
% \ref{prop:uni-1}
connects all fibers 
with  $(X_{+1}, \dots, X_{+J}) > 0$.
More strongly, the set of moves
is norm-reducing \citep{takemura-aoki-2005bernoulli}
for any two tables $x,y$ in any fiber with positive
marginals, i.e.\ we can make the $L_1$ distance between $x$ and $y$
smaller by a move from the set.

\begin{proposition} 
 \label{prop:uni-1}
 The set of moves
 \begin{equation}
  \label{eq:zijji2}
   {\cal B}_{\Lambda(A)} =\{ \pm (\Be_{j_1} + \Be_{j_4} - \Be_{j_2} -
   \Be_{j_3})  
   \mid 1\le j_1 < j_2 \le j_3 < j_4 \le J, \ \  j_2 - j_1=j_4 -j_3\}
 \end{equation}
 is norm-reducing for all fibers with $(X_{+1}, \dots,X_{+J}) > 0$ 
 for the univariate logistic
 regression model. % with $J$ levels. 
\end{proposition}

%The proof of this proposition is given by the Appendix.

To prove this proposition, we present a simple lemma.

\begin{lemma}
 \label{lemma:appendix-1}
 Let $z = \{z_{ij}\}$ be any move for the univariate logistic regression.
 Then there exist $j_1 < j_2$ and $j_3 < j_4$ satisfying the following
 conditions. 
 \begin{enumerate}
 \setlength{\itemsep}{0pt}
  \item[{\rm (a)}] $z_{1 j_1} > 0$, $z_{1 j_2} < 0$, $z_{1 j_3} < 0$, 
        $z_{1 j_4} > 0$ ;
  \item[{\rm (b)}] $z_{1 j_1} = 1$ implies $j_1 \neq j_4$ ;
  \item[{\rm (c)}] $z_{1 j_2} = -1$ implies $j_2 \neq j_3$ ;
  \item[{\rm (d)}] $z_{1j} = 0$ for $j_1 < j < j_2$ and $j_3 < j < j_4$.
 \end{enumerate}
\end{lemma}
\begin{proof}
 (a), (b) and (c) are obvious from the constraint 
 (\ref{linear-constraint-1}) and $z_{1+} = 0$.
 We can assume without loss of generality that there exist $j_1 < j_2$ 
 such that $z_{1 j_1} > 0$, $z_{1 j_2} < 0$, $z_{1 j} \ge 0$ for 
 $1 \le j < j_2$  and $z_{1j} = 0$ for $j_1 < j < j_2$.
 Since there exist $j_2 \le j_3 < j_4$ satisfying (a), (b) and (c), 
 we can choose $j_3$ and $j_4$ to satisfy (d). 
\end{proof}

We now give a proof of Proposition \ref{prop:uni-1}.

%\begin[Proof of Propsition \ref{prop:1}]{proof}
\begin{proof}[Proof of Propsition \ref{prop:uni-1}]
We employ the distance reducing argument of 
\cite{takemura-aoki-2005bernoulli}. 
Let $x$ and $y$ be two tables in the same fiber.
Then $z := x - y$ is a move. 
We can assume without loss of generality that there exist 
$j_1 < j_2 \le j_3 < j_4$ which satisfy the conditions of 
Lemma \ref{lemma:appendix-1} and $j_2 - j_1 \le j_4 - j_3$.
Define $j_5$  as $j_5 := j_4 - (j_2 - j_1)$. 
Then by applying a move
\[
 z' := - \bm{e}_{j_1} + \bm{e}_{j_2} + \bm{e}_{j_5} - \bm{e}_{j_4}, 
\]
we can reduce the $L_1$ distance between $x$ and $y$, because at least
one of the following operations can be performed to $x$ or $y$:
\[
 \begin{array}{rrrrr}
  & j_1 & j_2 & j_5 & j_4 \\ \cline{2-5}
   i=1 & \multicolumn{1}{|r}{+} & 0+ & 0+ & \multicolumn{1}{r|}{+} \\
  i=2 & \multicolumn{1}{|r}{0+} & + & + & \multicolumn{1}{r|}{0+} \\
\cline{2-5}
 \end{array}
\quad +
 \begin{array}{rrrrr}
  \\ \cline{2-5}
  & \multicolumn{1}{|r}{-1} & 1 & 1 & \multicolumn{1}{r|}{-1} \\
  & \multicolumn{1}{|r}{1} & -1  & -1 & \multicolumn{1}{r|}{1} \\
\cline{2-5}
 \end{array}
\quad = 
 \begin{array}{rrrrr}
  \\ \cline{2-5}
  & \multicolumn{1}{|r}{0+} & + &  + & \multicolumn{1}{r|}{0+} \\
  & \multicolumn{1}{|r}{+} & 0+  & 0+ & \multicolumn{1}{r|}{+} \\
\cline{2-5}
 \end{array}
\]

\[
 \begin{array}{rrrrr}
  & j_1 & j_2 & j_5 & j_4 \\ \cline{2-5}
   i=1 & \multicolumn{1}{|r}{0+} & + & + & \multicolumn{1}{r|}{0+} \\
  i=2 & \multicolumn{1}{|r}{+} & 0+ & 0+ & \multicolumn{1}{r|}{+} \\
\cline{2-5}
 \end{array}
\quad +
 \begin{array}{rrrrr}
  \\ \cline{2-5}
  & \multicolumn{1}{|r}{1} & -1 & -1 & \multicolumn{1}{r|}{1} \\
  & \multicolumn{1}{|r}{-1} & 1  & 1 & \multicolumn{1}{r|}{-1} \\
\cline{2-5}
 \end{array}
\quad = 
 \begin{array}{rrrrr}
  \\ \cline{2-5}
  & \multicolumn{1}{|r}{+} & 0+ &  0+ & \multicolumn{1}{r|}{+} \\
  & \multicolumn{1}{|r}{0+} & +  & + & \multicolumn{1}{r|}{0+} \\
\cline{2-5}
 \end{array}
\]
where $0+$ denote that the cell frequency is nonnegative.
\end{proof}

\bigskip
\cite{chen-dinwoodie-dobra-huber2005} introduced a subset of 
${\cal B}$ which still connects all fibers with $X_{+j} > 0, \forall j$. 
\cite{chen-dinwoodie-dobra-huber2005} did not 
give a proof of the following theorem.

\begin{theorem}[\cite{chen-dinwoodie-dobra-huber2005}]
 \label{theorem:univariate}
 The set of moves
 \begin{equation}
  \label{eq:zijji3}
   {\cal B}_0 =\{ z \in {\cal B} \mid j_2 = j_1+1, j_3 = j_4-1\}
 \end{equation}
 connects every  fiber satisfying $(X_{+1},\dots,X_{+J}) > 0$ 
 for the univariate logistic
 regression model. % with $J$ levels. 
\end{theorem}

%% In \cite{chen-dinwoodie-dobra-huber2005}, the proof of this theorem was
%% not presented and we give the proof in the Appendix.

\begin{proof}
It suffices to show that any move in $z \in {\cal B}_{\Lambda(A)}$ 
of  Proposition \ref{prop:uni-1}
can be
replaced by a series of moves in ${\cal B}_0$. To prove this 
it  suffices to show that the $L_1$ norm of any move
$z \in {\cal B}_{\Lambda(A)}$, i.e., the $L_1$ distance between the
positive part $x=z^+$ and the negative part $y=z^-$ of $z \in {\cal B}_{\Lambda(A)}$,
is reduced by moves in ${\cal B}_0$.
Denote $z := \Be_{j_1} - \Be_{j_2} - \Be_{j_3} + \Be_{j_4}$. 
We can assume without loss of generality that 
%$j_1 < j_2 < j_3 < j_4$.
$j_1 < j_2 \le j_3 < j_4$.
We prove it by the induction on 
$\delta := j_2 - j_1 = j_4 - j_3 \ge 2$.
%The proof is trivial for $\delta = 1$.
%Suppose $\delta_0 \geq 2$ and assume that the theorem holds 
%for $\delta \leq \delta_0$.

When $(x_{1,j_1+1},x_{1,j_4-1}) > 0$ or 
$(x_{2,j_1+1},x_{2,j_4-1}) > 0$, 
we can apply 
$z' := - \bm{e}_{j_1} + \bm{e}_{j_1+1} + \bm{e}_{j_4-1} - \bm{e}_{j_4}
$
to $z$ and   
$$
z + z' := 
\bm{e}_{j_1+1} - \bm{e}_{j_2} - \bm{e}_{j_3} + \bm{e}_{j_4-1}
$$
as seen from the picture below, where $z_{ij} = 0^*$ denotes that 
$x_{ij} = y_{ij} >0$.

\[
 \begin{array}{rrrrr}
  & j_1 & j_1 + 1 & j_4 - 1 & j_4 \\ \cline{2-5}
   i=1 & \multicolumn{1}{|r}{1} & 0^* & 0^*
   & \multicolumn{1}{r|}{1} \\
   i=2 & \multicolumn{1}{|r}{-1} & 0\hspace{0.15cm}  & 0\hspace{0.15cm}
    & \multicolumn{1}{r|}{-1} \\ 
\cline{2-5}
 \end{array}
\quad +
 \begin{array}{rrrrr}
  \\ \cline{2-5}
  & \multicolumn{1}{|r}{-1} & 1 & 1 & \multicolumn{1}{r|}{-1} \\
  & \multicolumn{1}{|r}{1} & -1  & -1 & \multicolumn{1}{r|}{1} \\
\cline{2-5}
 \end{array}
\quad = 
 \begin{array}{rrrrr}
  \\ \cline{2-5}
  & \multicolumn{1}{|r}{0} & 1 &  1 & \multicolumn{1}{r|}{0} \\
  & \multicolumn{1}{|r}{0} & -1  & -1 & \multicolumn{1}{r|}{0} \\
\cline{2-5}
 \end{array}
\]

\[
 \begin{array}{rrrrr}
  & j_1 & j_1 + 1 & j_4 - 1 & j_4 \\ \cline{2-5}
   i=1 & \multicolumn{1}{|r}{1} & 0\hspace{0.15cm} & 0\hspace{0.15cm}
   & \multicolumn{1}{r|}{1} \\
   i=2 & \multicolumn{1}{|r}{-1} & 0^*  & 0^* & \multicolumn{1}{r|}{-1} \\
\cline{2-5}
 \end{array}
\quad +
 \begin{array}{rrrrr}
  \\ \cline{2-5}
  & \multicolumn{1}{|r}{-1} & 1 & 1 & \multicolumn{1}{r|}{-1} \\
  & \multicolumn{1}{|r}{1} & -1  & -1 & \multicolumn{1}{r|}{1} \\
\cline{2-5}
 \end{array}
\quad = 
 \begin{array}{rrrrr}
  \\ \cline{2-5}
  & \multicolumn{1}{|r}{0} & 1 &  1 & \multicolumn{1}{r|}{0} \\
  & \multicolumn{1}{|r}{0} & -1  & -1 & \multicolumn{1}{r|}{0} \\
\cline{2-5}
 \end{array}
\]
From the inductive assumption, we can reduce 
the $L_1$ norm of $z+z'$ by moves in ${\cal B}_0$.

When $(x_{1,j_1+1},x_{2,j_3+1}) > 0$ or 
$(x_{2,j_1+1},x_{1,j_3+1}) > 0$, 
we can apply 
$z' := - \bm{e}_{j_1} + \bm{e}_{j_1+1} + \bm{e}_{j_3} - \bm{e}_{j_3+1}$
to $z$ and   
$$
z + z' := 
\bm{e}_{j_1+1} - \bm{e}_{j_2} - \bm{e}_{j_3+1} + \bm{e}_{j_4}
$$
as seen from the picture below. 
\[
 \begin{array}{rrrrr}
  & j_1 & j_1+1 & j_3 & j_3+1 \\ \cline{2-5}
   i=1 & \multicolumn{1}{|r}{1} & 0^* & -1
   & \multicolumn{1}{r|}{0\hspace{0.15cm}} \\
   i=2 & \multicolumn{1}{|r}{-1} & 0\hspace{0.15cm}  & 1 &
    \multicolumn{1}{r|}{0^*} \\ 
\cline{2-5}
 \end{array}
\quad +
 \begin{array}{rrrrr}
  \\ \cline{2-5}
  & \multicolumn{1}{|r}{-1} & 1 & 1 & \multicolumn{1}{r|}{-1} \\
  & \multicolumn{1}{|r}{1} & -1  & -1 & \multicolumn{1}{r|}{1} \\
\cline{2-5}
 \end{array}
\quad = 
 \begin{array}{rrrrr}
  \\ \cline{2-5}
  & \multicolumn{1}{|r}{0} & 1 & 0 & \multicolumn{1}{r|}{-1} \\
  & \multicolumn{1}{|r}{0} & -1  & 0 & \multicolumn{1}{r|}{1} \\
\cline{2-5}
 \end{array}
\]

\[
 \begin{array}{rrrrr}
  & j_1 & j_1+1 & j_3 & j_3+1 \\ \cline{2-5}
   i=1 & \multicolumn{1}{|r}{1} & 0\hspace{0.15cm} & -1
   & \multicolumn{1}{r|}{0^*} \\
   i=2 & \multicolumn{1}{|r}{-1} & 0^*  & 1 & \multicolumn{1}{r|}{0\hspace{0.15cm}} \\
\cline{2-5}
 \end{array}
\quad +
 \begin{array}{rrrrr}
  \\ \cline{2-5}
  & \multicolumn{1}{|r}{-1} & 1 & 1 & \multicolumn{1}{r|}{-1} \\
  & \multicolumn{1}{|r}{1} & -1  & -1 & \multicolumn{1}{r|}{1} \\
\cline{2-5}
 \end{array}
\quad = 
 \begin{array}{rrrrr}
  \\ \cline{2-5}
  & \multicolumn{1}{|r}{0} & 1 &  0 & \multicolumn{1}{r|}{-1} \\
  & \multicolumn{1}{|r}{0} & -1  & 0 & \multicolumn{1}{r|}{1} \\
\cline{2-5}
 \end{array}
\]
From the inductive assumption, we can reduce 
the $L_1$ norm of $z+z'$ by moves in ${\cal B}_0$.
In the case that 
$(x_{1,j_2-1},x_{2,i_4-1}) > 0$ or 
$(x_{2,j_2-1},x_{1,i_4-1}) > 0$, 
the proof is similar.

Suppose that 
$$
x_{2,j_1+1}=x_{1,j_2-1}=x_{2,j_3+1}=x_{1,j_4-1}=0.
$$
We note that this implies that 
$$
x_{1,j_1+1} > 0,\quad  
x_{2,j_2-1} > 0,\quad  
x_{1,j_3+1}> 0,\quad  
x_{2,j_4-1}> 0.
$$
Then there exists $j_3 < j_5 < j_4$ such that 
$$
x_{1 j_5} > 0,  \quad x_{2,j_5+1} > 0
$$
as in the following picture. 
\[
 \begin{array}{rrrrrrrrrrr}
  & j_1 & j_1+1 & j_2-1 & j_2 & j_3 & j_3 + 1 & j_5 & j_5+1 & j_4-1 &
   j_4 \\ \cline{2-11} 
   i=1 & \multicolumn{1}{|r}{1} & 0^* & 0\hspace{0.15cm} & -1 & -1 & 0^*
   & 0^* & 0\hspace{0.15cm} & 0\hspace{0.15cm} & \multicolumn{1}{r|}{1}\\ 
 i=2 & \multicolumn{1}{|r}{-1} & 0\hspace{0.15cm} & 0^* & 1
   & 1 & 0\hspace{0.15cm} & 0\hspace{0.15cm} & 0^* & 0^* & 
   \multicolumn{1}{r|}{-1}\\ 
    \cline{2-11}
 \end{array}
\]
By applying 
\[
 z' := - \bm{e}_{j_1} + \bm{e}_{j_1+1} + \bm{e}_{j_5} - \bm{e}_{j_5+1}
\]
and
\[
 z'' := - \bm{e}_{j_5} + \bm{e}_{j_5+1} + \bm{e}_{j_4-1} - \bm{e}_{j_4}
\]
to $z$ in this order, we obtain
\[
 z + z' + z'' = 
 \bm{e}_{j_1+1} - \bm{e}_{j_2} - \bm{e}_{j_3} - \bm{e}_{j_4-1}
\]
as in the following picture. 

\begin{align*}
 & \begin{array}{rrrrrrrrrrr}
  & j_1 & j_1+1 & j_2-1 & j_2 & j_3 & j_3 + 1 & j_5 & j_5+1 & j_4-1 &
   j_4 \\ \cline{2-11} 
   i=1 & \multicolumn{1}{|r}{1} & 0^* & 0\hspace{0.15cm} & -1 & -1 & 0^*
   & 0^* & 0\hspace{0.15cm} & 0\hspace{0.15cm} & \multicolumn{1}{r|}{1}\\ 
 i=2 & \multicolumn{1}{|r}{-1} & 0\hspace{0.15cm} & 0^* & 1
   & 1 & 0\hspace{0.15cm} & 0\hspace{0.15cm} & 0^* & 0^* & 
   \multicolumn{1}{r|}{-1}\\ 
    \cline{2-11}
 \end{array}\\
 & \quad + 
 \begin{array}{rrrrrrrrrrr}
  & j_1 & j_1+1 & j_2-1 & j_2 & j_3 & j_3 + 1 & j_5 & j_5+1 & j_4-1 &
   j_4 \\ \cline{2-11} 
   i=1 & \multicolumn{1}{|r}{-1} & 1 & 0\hspace{0.15cm} &  0 &  0 & 0
   & 1 & -1 & 0 & \multicolumn{1}{r|}{0}\\ 
   i=2 & \multicolumn{1}{|r}{1} & -1 & 0\hspace{0.15cm} &  0 &  0 & 0
   & -1 & 1 & 0 & \multicolumn{1}{r|}{0}\\ 
    \cline{2-11}
 \end{array}\\
 & \quad = 
 \begin{array}{rrrrrrrrrrr}
  & j_1 & j_1+1 & j_2-1 & j_2 & j_3 & j_3 + 1 & j_5 & j_5+1 & j_4-1 &
   j_4 \\ \cline{2-11} 
   i=1 & \multicolumn{1}{|r}{0} & 1 & 0\hspace{0.15cm} &  -1 &  1 & 0
   & 1 & -1 & 0 & \multicolumn{1}{r|}{-1}\\ 
   i=2 & \multicolumn{1}{|r}{0} & -1 & 0\hspace{0.15cm} &  1 &  -1 & 0
   & -1 & 1 & 0 & \multicolumn{1}{r|}{1}\\ 
    \cline{2-11}
 \end{array}
\end{align*}

\begin{align*}
 & \begin{array}{rrrrrrrrrrr}
  & j_1 & j_1+1 & j_2-1 & j_2 & j_3 & j_3 + 1 & j_5 & j_5+1 & j_4-1 &
   j_4 \\ \cline{2-11} 
   i=1 & \multicolumn{1}{|r}{0} & 1 & 0\hspace{0.15cm} &  -1 &  1 & 0
   & 1 & -1 & 0 & \multicolumn{1}{r|}{-1}\\ 
   i=2 & \multicolumn{1}{|r}{0} & -1 & 0\hspace{0.15cm} &  1 &  -1 & 0
   & -1 & 1 & 0 & \multicolumn{1}{r|}{1}\\ 
    \cline{2-11}
 \end{array}\\
 & \quad + 
 \begin{array}{rrrrrrrrrrr}
  & j_1 & j_1+1 & j_2-1 & j_2 & j_3 & j_3 + 1 & j_5 & j_5+1 & j_4-1 &
   j_4 \\ \cline{2-11} 
   i=1 & \multicolumn{1}{|r}{0} & 0 & 0\hspace{0.15cm} &  0 &  0 & 0
   & -1 & 1 & 1 & \multicolumn{1}{r|}{-1}\\ 
   i=2 & \multicolumn{1}{|r}{0} & 0 & 0\hspace{0.15cm} &  0 &  0 & 0
   & 1 & -1 & -1 & \multicolumn{1}{r|}{1}\\ 
    \cline{2-11}
 \end{array}\\
 & \quad = 
\begin{array}{rrrrrrrrrrr}
  & j_1 & j_1+1 & j_2-1 & j_2 & j_3 & j_3 + 1 & j_5 & j_5+1 & j_4-1 &
   j_4 \\ \cline{2-11} 
   i=1 & \multicolumn{1}{|r}{0} & 1 & 0\hspace{0.15cm} &  -1 &  1 & 0
   & 0 & 0 & -1 & \multicolumn{1}{r|}{0}\\ 
   i=2 & \multicolumn{1}{|r}{0} & -1 & 0\hspace{0.15cm} &  1 &  -1 & 0
   & 0 & 0 & 1 & \multicolumn{1}{r|}{0}\\ 
    \cline{2-11}
 \end{array}
\end{align*}

From the inductive assumption, we can reduce 
the $L_1$ norm of $z+z'$ by moves in ${\cal B}_0$.
In the case that 
$x_{1,j_1+1}=x_{2,j_2-1}=x_{1,j_3+1}=x_{2,j_4-1}=0$, 
the proof is similar.
\end{proof}

\section{Markov bases for models of Segre product type}
\label{sec:Segre}

In the previous section we considered univariate Poisson regression
and logistic regression.  We now consider generalizing the results to
multiple regression.  In this section we show an explicit form of
Markov basis for multiple Poisson regression.  Therefore an extension
of Proposition  \ref{prop:1} to multiple regression is straightforward.
In contrast,  as we see in the next section, it is difficult to generalize 
the results of univariate logistic regression to multiple logistic regression.

Multiple Poisson regression is a Segre product of univariate Poisson
regressions. Quadratic Gr\"obner bases of Segre products is already 
discussed in \cite{ahot2008}.  However Theorem \ref{thm:segre} below is 
about Markov bases (rather than
Gr\"obner bases) and it is applicable even if the component
configurations do not possess quadratic  Gr\"obner bases.

Consider two configurations $A=(\Ba_1, \dots, \Ba_J)$ and
$B=(\Bb_1, \dots, \Bb_K)$, where $\Ba_j$ and $\Bb_k$ are column vectors.
We assume the homogeneity, i.e., 
there exist weight vectors $w,v$ such that 
$
\langle w, \Ba_j \rangle = 1, \ \forall j, \ 
\langle v, \Bb_k \rangle = 1, \ \forall k.
$
% Homogeneity guarantees that each fiber has the same ``sample size'' (total degree).
The configuration $A\otimes B$ of the Segre product of $A$ and $B$ is defined as
\[
A\otimes B = \Big(\Ba_j \oplus \Bb_k, \ j=1,\dots,J, k=1,\dots,K\Big),  \qquad
\Ba_j \oplus \Bb_k =\begin{pmatrix} \Ba_j \\ \Bb_k
\end{pmatrix}.
\]
% Here I am thinking of Segre product ``$A\otimes B$'' of two
% configurations.  For the product of three or more configurations, we
% can use the associativity $A\otimes B \otimes C= (A\otimes B)\otimes
% C$.  Therefore (at least for the purpose of this memo) consideration of
% product of two configurations is essential.
If both $A$ and $B$ are configurations of the form 
(\ref{eq:configuratio-univariate-Poisson})
for the univariate Poisson regression model, then $A\otimes B$ corresponds
to the bivariate Poisson regression model, where $X_{jk}$ is independently distributed
according to Poisson distribution  with mean $\mu_{jk}$, which is modeled as
\[
\log \mu_{jk}= \mu + \alpha j + \beta k, \qquad j=1,\dots,J,\ k=1,\dots,K.
\]

Let $X=(X_{jk})_{j=1,\dots,J, k=1,\dots,K}$ denote a table of observed 
frequencies.
%% The row sum is denoted by  $x_{j+}=\sum_k x_{jk}$ and the column sum
%% is denoted $x_{+k}=\sum_j x_{jk}$.  
The sufficient statistic for the
Segre product $A\otimes B$ is given by
\[
\sum_j \Ba_j X_{j+}, \quad \sum_k \Bb_k X_{+k}.
\]
Therefore $z=(z_{jk})$ is a move for $A\otimes B$ if and only if
\begin{equation}
\label{eq:move}
0=\sum_j \Ba_j z_{j+}, \quad 0=\sum_k \Bb_k z_{+k}.
\end{equation}

Given Markov bases ${\cal B}_A$ and 
${\cal B}_B$ for $A$ and $B$, respectively, our goal is to 
construct a Markov basis for the Segre product $A\otimes B$.
Denote the elements of ${\cal B}_A$ by $z^A=(z^A_1, \dots, z^A_J)$.
% and
% denote the elements of  ${\cal B}_B$ by $z^B=(z^B_1, \dots, z^B_I)$.
Let $z_j^{A,+}=\max(z_j^A,0)$ be the positive part and
$z_j^{A,-}=\max(-z_j^A,0)$ be the negative part of $z_j^A$.  Let
$\deg z^A=\sum_{j=1}^J z_j^{A,+}=\sum_{j=1}^J z_j^{A,-}$ be the degree of
$z^A$. 
Now %e can choose 
$z^A$ is uniquely written as
\[
z^A = \sum_{h=1}^{\deg z^A} (\Be_{j_h} - \Be_{j'_h}), 
\]
where $j_1\le \dots \le j_{\deg z^A}$ and $j'_1 \le \dots \le j'_{\deg
z^A}$. 
Let $\Be_{j k}$ denote a $J\times K$ table with $1$ at the cell $(j,k)$ and
0 everywhere else.
Now choose arbitrary $1\le k_1,  \dots,k_{\deg z^A} \le K$ and define
\[
z^A(k_1, \dots, k_{\deg z^A})=
\sum_{h=1}^{\deg z^A} (\Be_{j_h k_h} - \Be_{j'_h k_h}).
\]
We call $z^A(k_1, \dots, k_{\deg z^A})$ a ``distribution'' of $z^A$ by 
coordinates $k_1, \dots, k_{\deg z^A}$.
Note that $k_1, \dots, k_{\deg z^A}$ are not ordered.
Similarly define the distribution 
$z^B(j_1, \dots, j_{\deg z^B})$ of a move $z^B \in {\cal B}_B$.

In addition to these moves we also consider the basic moves $z(j_1, j_2; k_1, k_2)
=\Be_{j_1  k_1}+\Be_{j_2 k_2}-\Be_{j_1 k_2}-\Be_{j_2 k_1}$
of the form
{\footnotesize
$
 \begin{array}{cc}
     +1 & -1\\  -1 & +1
  \end{array} .
$}
We now have the following theorem.

\begin{theorem}
\label{thm:segre}
 The set of basic moves 
 %for the complete independence model of 
 %$J_1 \times\dots \times J_m$ 
 %$J \times K$ contingency tables 
 and the set of moves of the form 
 $z^A(k_1, \dots, k_{\deg z^A})$, $1\le k_1,  \dots,k_{\deg z^A} \le K$, 
 $z^A\in {\cal B}_A$, 
 $z^B(j_1, \dots, j_{\deg z^B})$, $1\le  j_1,  \dots, j_{\deg z^B} \le J$, 
 $z^B\in {\cal B}_B$, form a Markov basis for the Segre product $A\otimes B$.
\end{theorem}

A proof of this theorem is given in Appendix.
In Theorem \ref{thm:segre}
we have considered Segre product of two configurations.
By a recursive argument,  a Markov basis for
the Segre product of
arbitrary number of configurations $A_1 \otimes\dots\otimes A_m$
is given as follows.  Let ${\cal B}_{A_j}$ be a Markov basis 
for the configuration $A_j$, $j=1,\dots,m$.  Write $[J]=\{1,\dots,J\}$ and
let 
\[
\bar {\cal J}_j=[J_1] \times \dots\times [J_{j-1}] \times [J_{j+1}]\times \dots
\times [J_m].
\]
Let  $z^{A_j} \in {\cal B}_{A_j}$ and let
$\Bk_1, \dots,\Bk_{\deg z^{A_j}}\in \bar {\cal J}_j$.  Now define
\[
z^{A_j}(\Bk_1, \dots, \Bk_{\deg z^{A_j}})=
\sum_{h=1}^{\deg z^{A_j}} (\Be_{j_h \Bk_h} - \Be_{j'_h \Bk_h}),
\]
where $\Be_{j,\Bk}$ is an $m$-way table with 1 at
the cell $(j,\Bk)$ and 0 everywhere else.  Then we have the following
corollary to Theorem \ref{thm:segre}.

\begin{corollary}
The set of %basic 
 square-free  degree two moves for the complete independence model of 
$J_1 \times\dots \times J_m$ contingency tables 
% The set of basic moves 
and the set of moves of the form 
$z^{A_j}(\Bk_1, \dots, \Bk_{\deg z^A})$, 
$\Bk_1, \dots, \Bk_{\deg z^A}\in \bar {\cal J}_j$, $j=1,\dots,m$, 
form a Markov basis for the Segre product $A_1\otimes \dots \otimes A_m$.
\end{corollary}

Minimality of the Markov basis constructed in Theorem \ref{thm:segre}
is not clear at the present.  However the maximum degree of moves in
the Markov basis for $A_1 \otimes \dots \otimes A_m$ is bounded by the
maximum degree of moves in ${\cal B}_{A_1}, \dots {\cal B}_{A_m}$.

\section{Connectivity of fibers of positive marginals in bivariate
  logistic regression}
\label{sec:connectivity}

In this section we consider the extension of the results in univariate
logistic regression model to bivariate logistic regression model.  
Let $\{1,\dots, J\}$ and $\{1,\dots, K\}$ 
be the sets levels of two covariates. 
Let $X_{1jk}$ and $X_{2jk}$, $j=1,\ldots,J$, $k=1,\ldots,K$, 
be the numbers of successes and failures, respectively, for level $(j,k)$.
The probability for success $p_{1jk}$ is modeled as
\begin{equation}
 \label{model:bivariate}
 \mathrm{logit}(p_{1jk}) = 
 \log \left(
 \frac{p_{1jk}}{1-p_{1jk}}
 \right)
 =
 \mu + \alpha j + \beta k, 
\end{equation}
\[
 j = 1, \ldots, J, \quad k = 1, \ldots, K.
\]
% The likelihood is written by 
% \begin{align*}
%  L(\alpha, \beta, \gamma) & \propto
%  \prod_{j=1}^J \prod_{k=1}^K 
%  \left(
%  1 + \exp (\alpha + \beta j + \gamma k )
%  \right)^{-n_{+jk}}\\
%  & \qquad \times
%  \prod_{j=1}^J \prod_{k=1}^K 
%  \exp \left(
%  \alpha  n_{1jk} + \beta  j  n_{1jk} + \gamma  k
%   n_{1jk}
%  \right)\\
%  & = 
%  \prod_{j=1}^J \prod_{k=1}^K 
%  \left(
%  1 + \exp (\alpha + \beta j + \gamma k )
%  \right)^{-n_{+jk}}\\
%  & \qquad \times
%  \exp \left(
%  \alpha  n_{1++} + \beta  \sum_{j=1}^J j n_{1j+} 
%  + \gamma \sum_{k=1}^K k n_{1+k}
%  \right)\\
% \end{align*}
The sufficient statistics for this model is 
$X_{1++}$, $\sum_{j=1}^J j X_{1j+}$,
$\sum_{k=1}^K k X_{1+k}$, $X_{+jk}$, $\forall j,k$. 
Hence moves $Z = (z_{ijk})$ for the model satisfy 
\[
z_{1++}=0, \quad \sum_{j=1}^J j z_{1j+}=0, \quad 
\sum_{k=1}^K k z_{1+k} = 0, \quad z_{+jk} = 0, \ \forall j,k.
\]
Let 
\[
B = \begin{pmatrix} 1& 1 & \dots & 1 \\
                    1 & 2 & \dots & K \\
                  \end{pmatrix}.
\]
and let $A$ be defined as in (\ref{def:A}). 
Then the configuration for the bivariate logistic regression model is 
the Lawrence lifting of Segre product $\Lambda(A \otimes B)$.
Here we consider a set of moves which connects every fiber satisfying
$X_{+jk} > 0$, $\forall j,k$.  

\begin{definition}
 \label{def:bivariate}
 Let $\bm{e}_{jk}=(e_{ijk})$ be redefined as an integer array 
 with $1$ at the cell $(1jk)$, $-1$ at the cell $(2jk)$ 
 and $0$ everywhere else.  
 Define ${\cal B}_{\Lambda(A \otimes B)}$ 
 as the set of moves $z=(z_{ijk})$ satisfying the following conditions, 
 \begin{enumerate}
  \item $z = \bm{e}_{j_1 k_1} - \bm{e}_{j_2 k_2} - \bm{e}_{j_3 k_3}
        + \bm{e}_{j_4 k_4}${\rm;}
  \item $(j_1,k_1) - (j_2,k_2) = (j_3,k_3) - (j_4,k_4)$.
 \end{enumerate}
\end{definition}

${\cal B}_{\Lambda(A \otimes B)}$ is an extension of ${\cal
B}_{\Lambda(A)}$ in Proposition \ref{prop:uni-1} to the bivariate 
model (\ref{model:bivariate}). 
We note that the $(i=1)$-slice of a moves $(z_{1jk})$ in 
${\cal B}_{\Lambda(A \otimes B)}$ is a move of the Markov basis defined
in Theorem \ref{thm:segre}. 
Now we present the main theorem of this paper.

\begin{theorem}
 \label{theorem:bivariate}
 ${\cal B}_{\Lambda(A \otimes B)}$ connects every fiber satisfying
 $X_{+jk} > 0$, $\forall j,k$.  
\end{theorem}

A proof of this theorem is given in Appendix.
%
%% \begin{example}
We give some examples of moves in ${\cal B}_{\Lambda(A \otimes B)}$.

\medskip

\begin{tabular}{lll}
  {\rm (1)} $k_1 = \cdots = k_4$ & 
  {\rm (2)} $k_1 = \cdots = k_4$ and $j_2 = j_3$ & 
  {\rm (3)} $k_1 = k_2$\  $(k_3 = k_4)$ \\
  \begin{tabular}{crrrr}
   & $j_1$ & $j_2$ & $j_3$ & $j_4$\\ \cline{2-5}
   $k_1$ & 
       \multicolumn{1}{|c}{$1$} & $-1$ & $-1$ & 
                   \multicolumn{1}{c|}{$1$} \\ \cline{2-5}
  \end{tabular}&
        \begin{tabular}{crrr}
         & $j_1$ & $j_2$ & $j_4$\\ \cline{2-4}
         $k_1$ & \multicolumn{1}{|c}{$1$} & $-2$ & 
                     \multicolumn{1}{c|}{$1$} \\ \cline{2-4}
        \end{tabular} & 
        \begin{tabular}{crrrr}
         & $j_1$ & $j_2$ & $j_3$ & $j_4$\\ \cline{2-5}
         $k_1$ & \multicolumn{1}{|c}{$1$} 
             & $-1$ & $0$ & 
                         \multicolumn{1}{c|}{$0$} \\
         $k_3$ & \multicolumn{1}{|c}{$0$} &  $0$ & $-1$ & 
                         \multicolumn{1}{c|}{$1$} \\ \cline{2-5}
        \end{tabular}\\
  \\
  {\rm (4)} $k_1 = k_2$ and $j_2 = j_3$ & 
  {\rm (5)} $(j_2,k_2)=(j_3,k_3)$ & 
  {\rm (6)} $k_1=k_4$ and $j_2=j_3$\\
        \begin{tabular}{crrr}
         & $j_1$ & $j_2$ & $j_4$\\ \cline{2-4}
         $k_1$ & \multicolumn{1}{|c}{$1$} & $-1$ & \multicolumn{1}{c|}{$0$} \\
         $k_3$ & \multicolumn{1}{|c}{$0$} & $-1$ & 
                     \multicolumn{1}{c|}{$1$}\\ \cline{2-4}
        \end{tabular} &
        \begin{tabular}{crrr}
         & $j_1$ & $j_2$ & $j_4$\\ \cline{2-4}
         $k_1$ & \multicolumn{1}{|c}{$1$} &  $0$ & \multicolumn{1}{c|}{$0$} \\
         $k_2$ & \multicolumn{1}{|c}{$0$} & $-2$ & \multicolumn{1}{c|}{$0$}\\ 
         $k_4$ & \multicolumn{1}{|c}{$0$} &  $0$ & \multicolumn{1}{c|}{$1$}\\ 
         \cline{2-4}
        \end{tabular} &
        \begin{tabular}{crrr}
         & $j_1$ & $j_2$ & $j_4$\\ \cline{2-4}
         $k_2$ & \multicolumn{1}{|c}{$0$} & $-1$ & \multicolumn{1}{c|}{$0$} \\
         $k_1$ & \multicolumn{1}{|c}{$1$} & $0$ & \multicolumn{1}{c|}{$1$}\\ 
         $k_3$ & \multicolumn{1}{|c}{$0$} & $-1$ & \multicolumn{1}{c|}{$0$}\\
         \cline{2-4} 
        \end{tabular}
 \end{tabular}
%\end{example} 

\section{Numerical examples}
\label{sec:examples}
\subsection{Data on coronary heart disease incidence}
\begin{table}[htbp]
 \caption{Data on coronary heart disease incidence}
 \begin{tabular}{ccccccccc} \hline
  & & \multicolumn{7}{c}{Serum Cholesterol (mg/100ml)}\\ \cline{3-9}
  & Blood & 1 & 2 & 3 & 4 & 5 & 6 & 7 \\ \cline{3-9}
  & Pressure & $<$ 200 & 200-209 & 210-219 & 220-244 & 245-259 
  & 260-284 & $>$ 284\\ \hline
  1 & $<$ 117 & 2/53 & 0/21 & 0/15 & 0/20 & 0/14 & 1/22 & 0/11\\
  2 & 117-126 & 0/66 & 2/27 & 1/25 & 8/69 & 0/24 & 5/22 & 1/19\\
  3 & 127-136 & 2/59 & 0/34 & 2/21 & 2/83 & 0/33 & 2/26 & 4/28\\
  4 & 137-146 & 1/65 & 0/19 & 0/26 & 6/81 & 3/23 & 2/34 & 4/23\\
  5 & 147-156 & 2/37 & 0/16 &  0/6 & 3/29 & 2/19 & 4/16 & 1/16\\
  6 & 157-166 & 1/13 & 0/10 & 0/11 & 1/15 & 0/11 & 2/13 & 4/12\\
  7 & 167-186 & 3/21 & 0/5 & 0/11 & 2/27 & 2/5 & 6/16 & 3/14\\
  8 & $>$ 186 & 1/5 & 0/1 & 3/6 & 1/10 & 1/7 & 1/7 & 1/7 \\ \hline
 \end{tabular}\\
 {\em Source} : \cite{cornfield1962}
 \label{table:Agresti}
\end{table}

Table \ref{table:Agresti} refers to coronary heart disease incidence in
Framingham, Massachusetts \citep{cornfield1962, agre:1990}. 
A sample of male residents, aged 40 through 50, were classified
on blood pressure and serum cholesterol concentration.
$2/53$ in the (1,1) cell means that there are 53 cases, of whom 2 exhibited
heart disease. 
We examine the goodness-of-fit of the model (\ref{model:bivariate})
with $J = 7$ and $K = 8$.  
We first test the null hypotheses $\mathrm{H}_\alpha : \alpha = 0$ and
$\mathrm{H}_\beta : \beta = 0$ versus (\ref{model:bivariate})   
using the likelihood ratio statistics $L_{\alpha}$ and $L_{\beta}$.
Then we have $L_{\alpha}=18.09$ and $L_{\beta}=22.56$ and  
the asymptotic p-values are $2.107 \times 10^{-5}$ and $2.037 \times
10^{-6}$, respectively, from the asymptotic distribution $\chi_1^2$. 
We computed the exact distribution of $L_{\alpha}$ and $L_{\beta}$ 
via Monte Carlo Markov chain (MCMC) with the sets of moves ${\cal B}_{\Gamma(A)}$ and ${\cal B}_0$
discussed in Section \ref{sec:preliminaries}. 
See the last paragraph of Section  \ref{sec:discussions} on sampling under
$H_\alpha$ and $H_\beta$. In all experiments in this paper, we sampled
100,000 tables after  
50,000 burn-in steps.
Figure \ref{fig:ag-1dim-a} and \ref{fig:ag-1dim-b} represent 
histograms of $L_{\alpha}$ and $L_{\beta}$. 
The solid lines in the figures represent the density function of 
the asymptotic distribution $\chi_1^2$.
The estimated p-values are $1.0 \times 10^{-6}$ for all cases. 
Therefore both $\mathrm{H}_{\alpha}$ and $\mathrm{H}_{\beta}$ are
rejected. 
We can see from the figures that there are little differences between  
two histograms computed with ${\cal B}_{\Gamma(A)}$ and ${\cal B}_0$.

Next we set (\ref{model:bivariate}) as a null hypothesis and test it
versus the following ANOVA type logit model, 
\begin{equation}
 \label{model:ANOVA}
 \mathrm{H}_1 : 
  \mathrm{logit}(p_{1jk}) = 
  \log \left(
  \frac{p_{1jk}}{1-p_{1jk}}
  \right)
  =
  \mu + \alpha_j + \beta_k, 
\end{equation}
where 
$\sum_{j=1}^J \alpha_j = 0$ and 
$\sum_{k=1}^K \beta_k = 0$ by likelihood ratio statistic $L_0$.
The value of $L_0$ is $13.07587$ and the asymptotic p-value is $0.2884$
from the asymptotic distribution $\chi^2_{11}$.  
We computed the exact distribution of $L_0$ via MCMC with 
${\cal B}_{\Gamma(A \otimes B)}$ defined in Definition
\ref{def:bivariate}. 
As an extension of ${\cal B}_0$ in Theorem \ref{theorem:univariate} to
the bivariate model (\ref{model:bivariate}), 
we define ${\cal B}_0^2$ by the set of moves 
$z = \bm{e}_{j_1 k_1} - \bm{e}_{j_2 k_2} - \bm{e}_{j_3 k_3} +
\bm{e}_{j_4 k_4}$ satisfying 
$(j_1,k_1) - (j_2,k_2) = (j_3,k_3) - (j_4,k_4)$ is either of 
$(\pm 1,0)$, $(0, \pm 1)$, $(\pm 1, \pm 1)$ or $(\pm 1, \mp 1)$.
We also computed the exact distribution of $L_0$ with ${\cal B}_0^2$. 
Figure \ref{fig:ag-a} represents histograms of $L_0$ computed with 
${\cal B}_{\Gamma(A \otimes B)}$ and ${\cal B}_0^2$. 
The estimated p-values are $0.2706$ with ${\cal B}_{\Gamma(A \otimes
B)}$ and $0.2958$ with ${\cal B}_0^2$. Therefore the model
(\ref{model:bivariate}) is accepted.

The p-values estimated with ${\cal B}_{\Gamma(A \otimes B)}$ and 
${\cal B}_0^2$ are close and there are little differences between two
histograms.    
From the results of Theorem \ref{theorem:univariate} and this
numerical experiment, ${\cal B}_0^2$ is also expected to connect every
fiber with positive response variable marginals.
However the theoretical proof of it is not clear at the present and 
is left to our future research. 

\begin{figure}[t]
 \centering
 \begin{tabular}{cc}
 \includegraphics[scale=0.25]{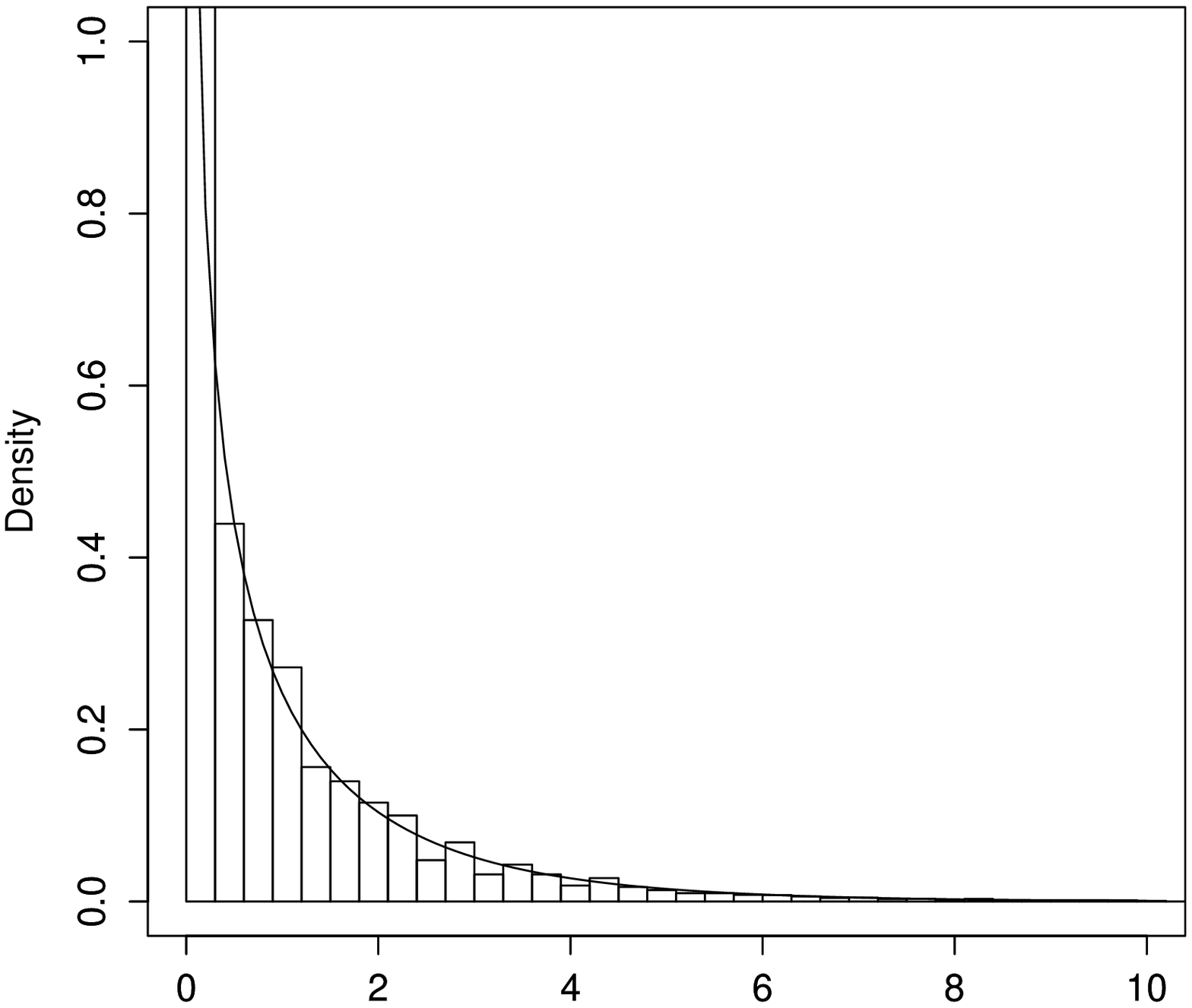} &
  \includegraphics[scale=0.25]{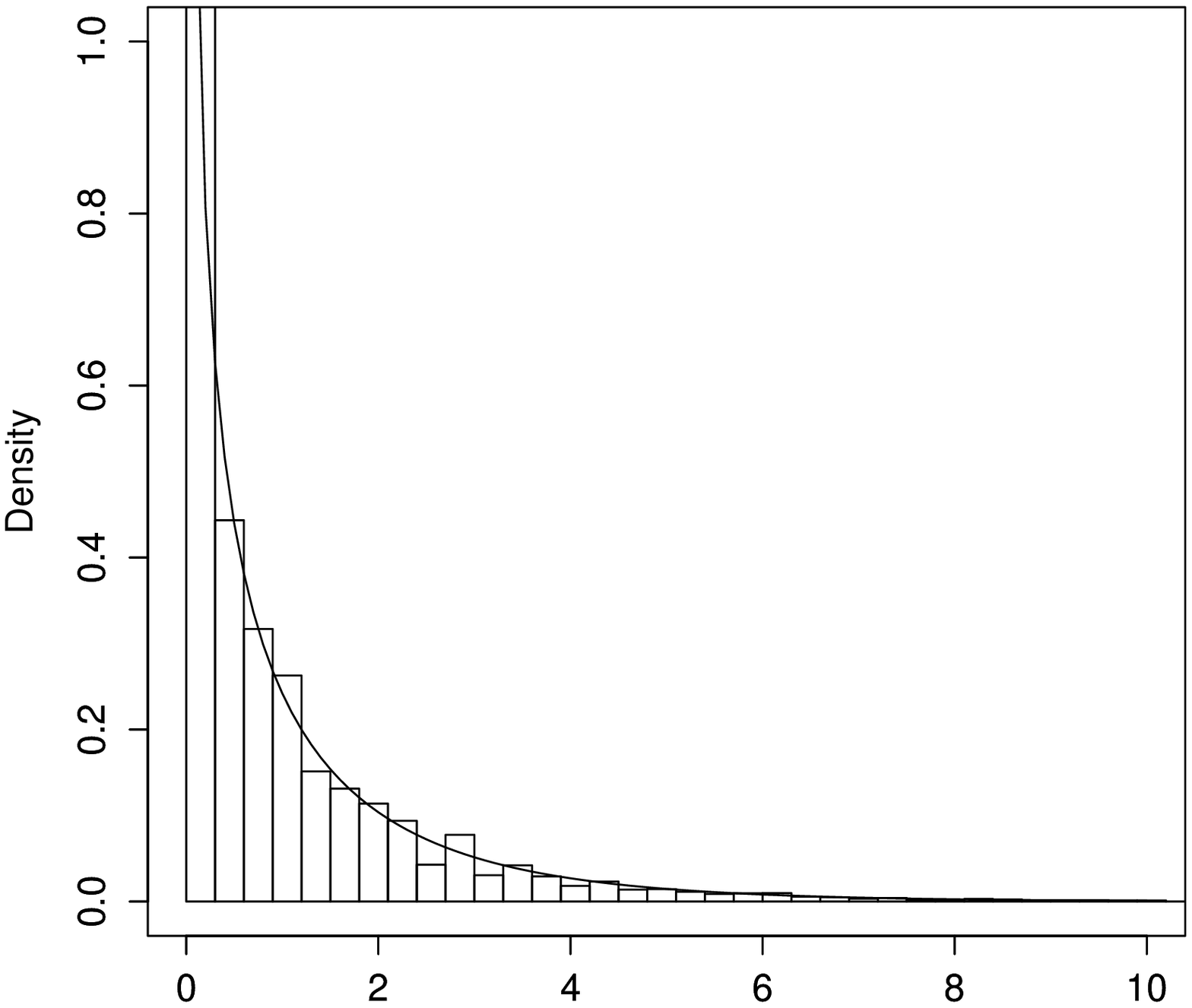}\\
  (a) A histogram with ${\cal B}_{\Gamma(A)}$ & 
  (b) A histogram with ${\cal B}_0$
 \end{tabular}
 \caption{Histograms of $L_{\alpha}$ via MCMC with 
 ${\cal B}_{\Gamma(A)}$ and ${\cal B}_0$} 
 \label{fig:ag-1dim-a}
 \vspace{0,5cm}
 \begin{tabular}{cc}
 \includegraphics[scale=0.25]{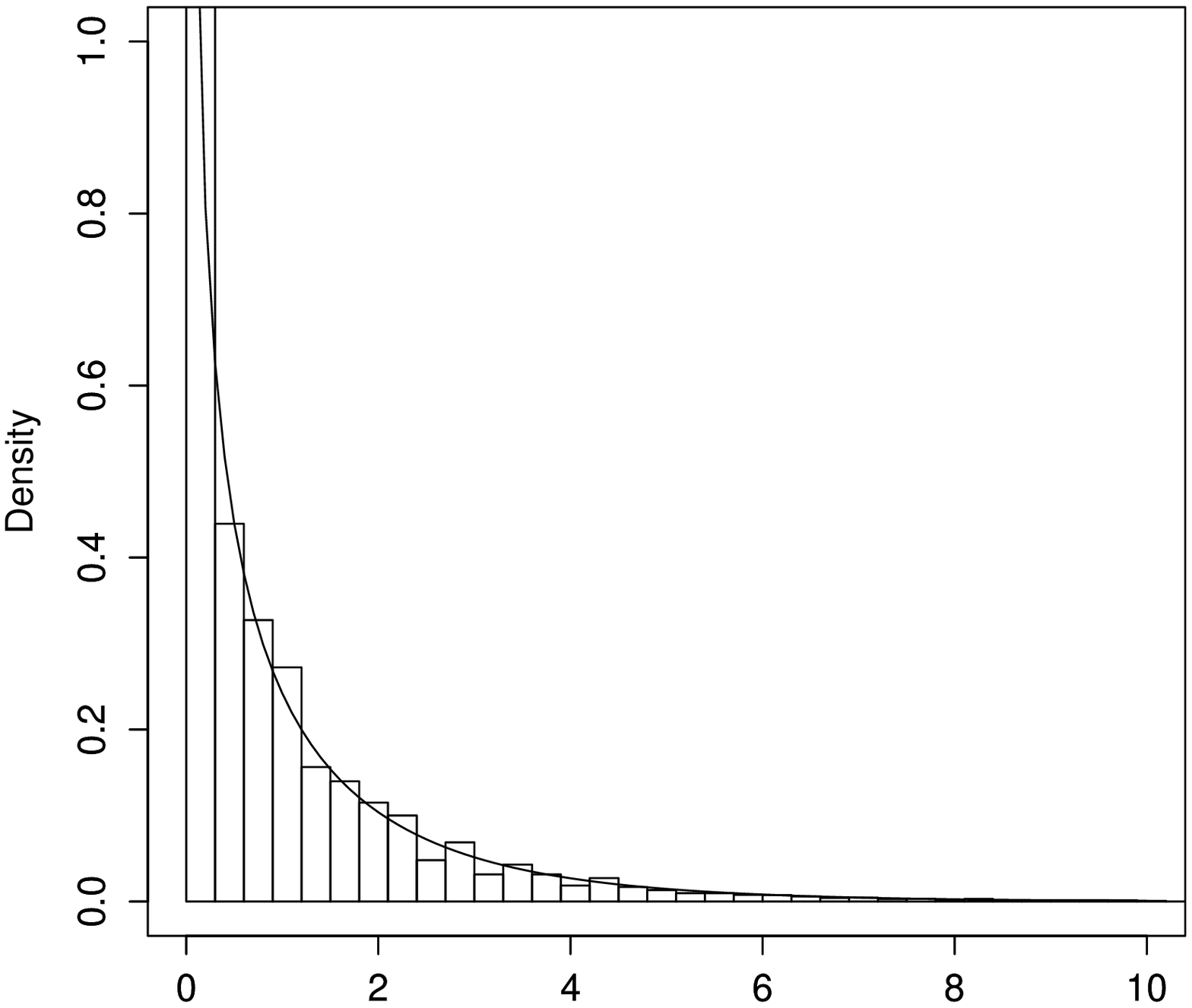} &
  \includegraphics[scale=0.25]{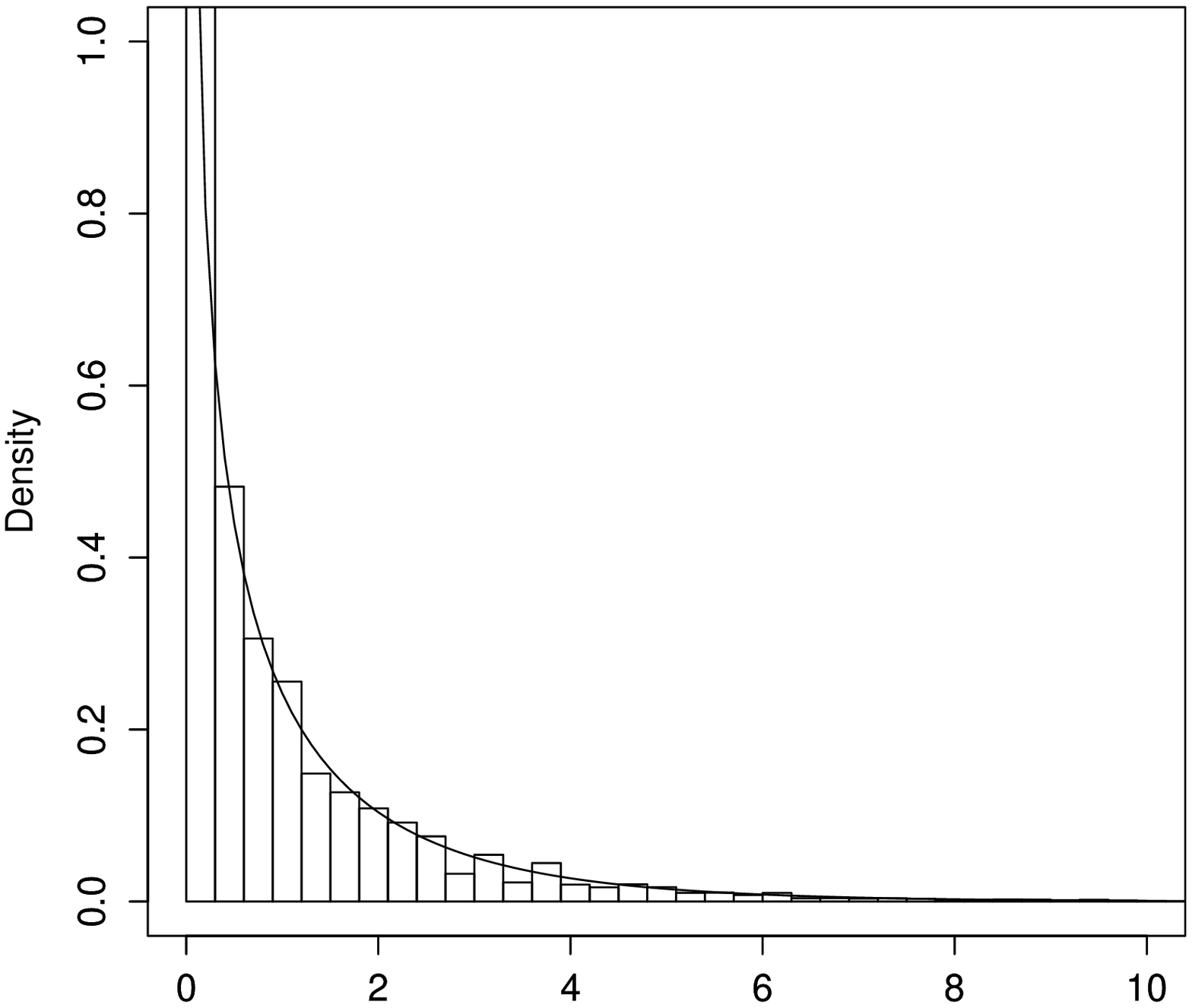}\\
  (a) A histogram with ${\cal B}_{\Gamma(A)}$ & 
  (b) A histogram with ${\cal B}_0$
 \end{tabular}
 \caption{Histograms of $L_{\beta}$ via MCMC with 
 ${\cal B}_{\Gamma(A)}$ and ${\cal B}_0$} 
 \label{fig:ag-1dim-b}
 \vspace{0,5cm}
 \begin{tabular}{cc}
 \includegraphics[scale=0.25]{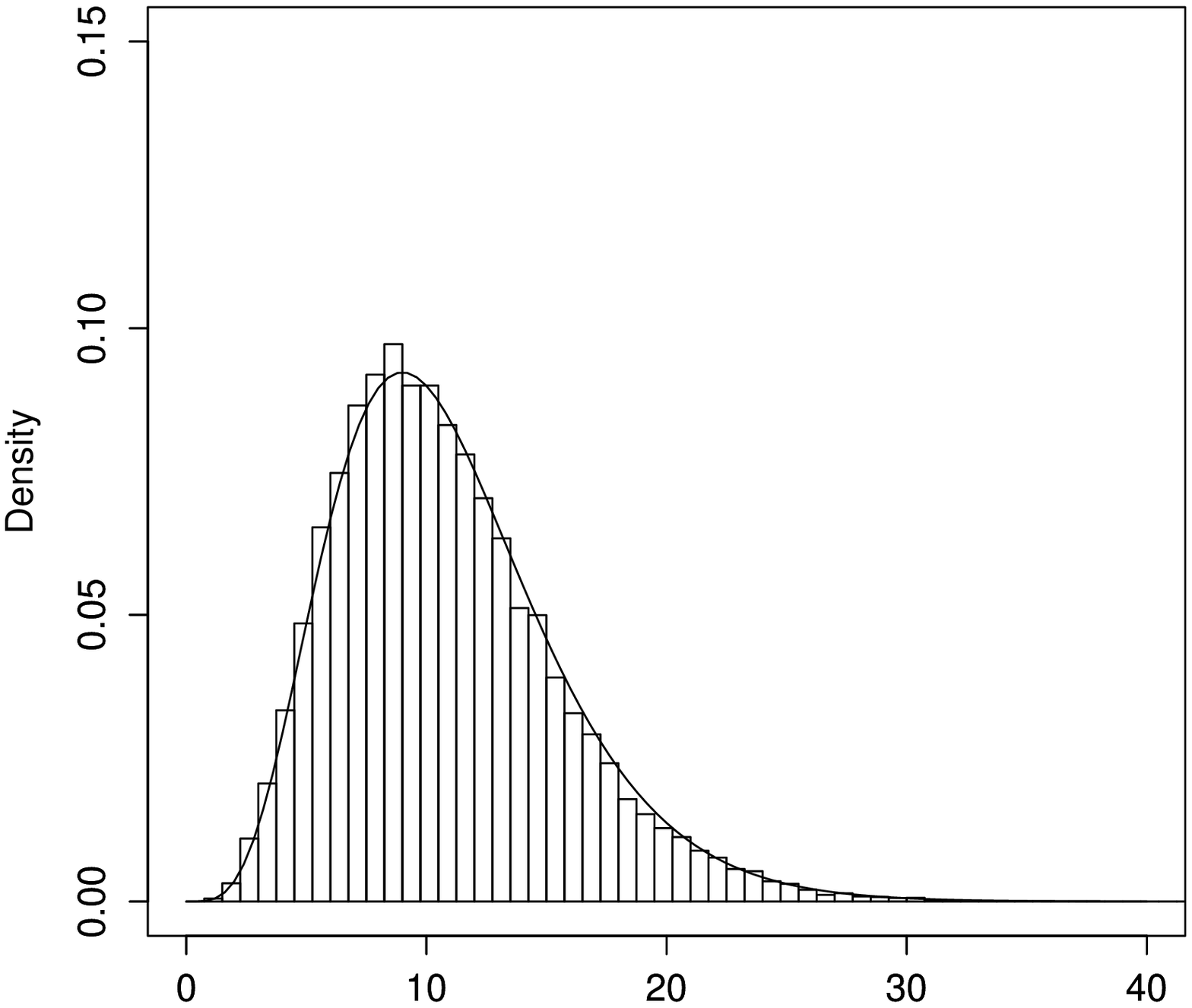} &
  \includegraphics[scale=0.25]{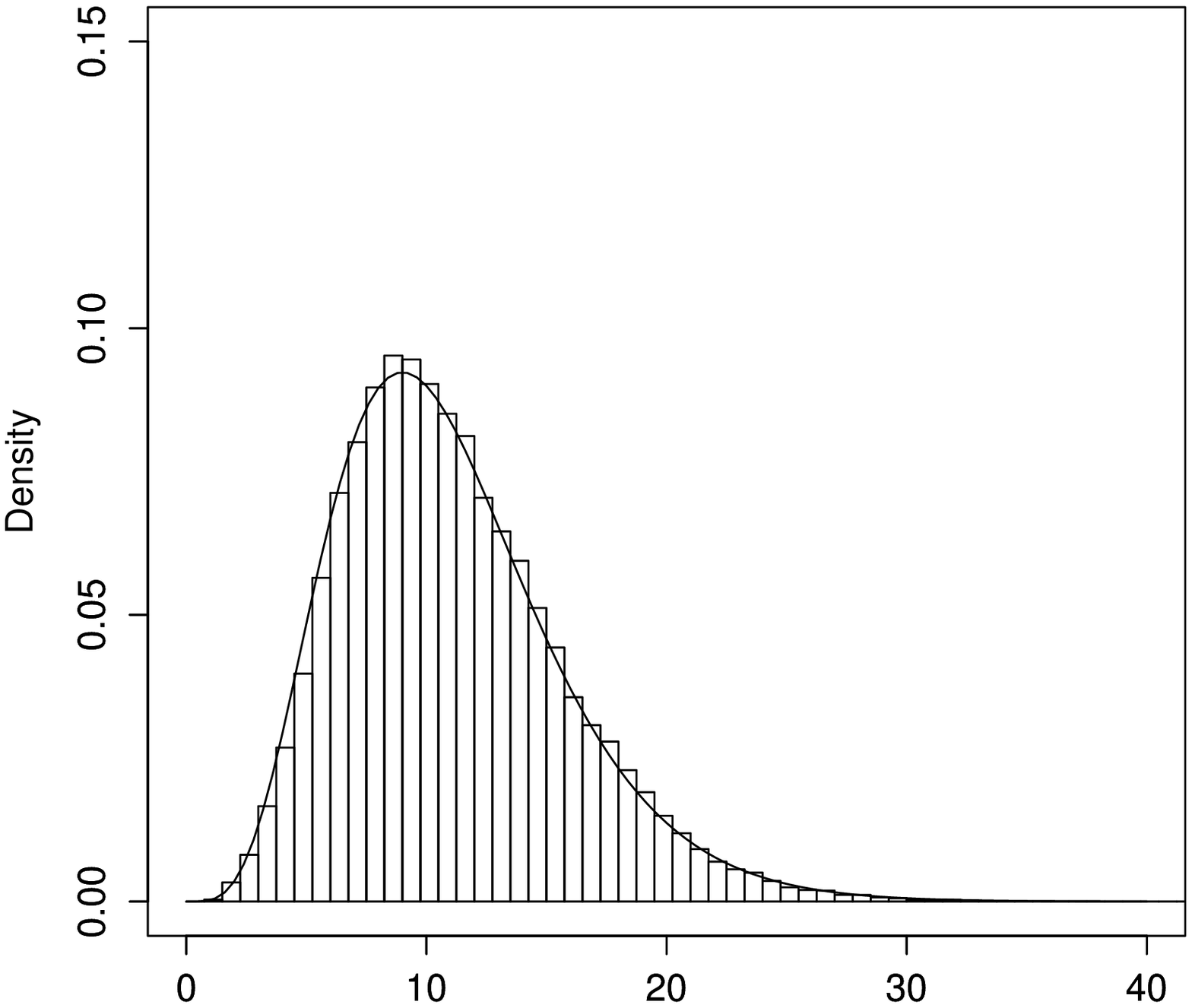}\\
  (a) A histogram with ${\cal B}_{\Lambda(A \otimes B)}$ & 
  (b) A histogram with ${\cal B}_0^2$
 \end{tabular}
 \caption{Histograms of $L_0$ via MCMC with ${\cal B}_{\Lambda(A \otimes
 B)}$ and ${\cal B}_0^2$}  
 \label{fig:ag-a}
\end{figure}

\subsection{Data on occurrence of esophageal cancer}
\begin{table}[htbp]
\centering
 \caption{Data on occurrence of esophageal cancer}
\begin{tabular}{cccccccc} \hline
  & & \multicolumn{6}{c}{Age}\\ \cline{3-8}
  & Alcohol & 1 & 2 & 3 & 4 & 5 & 6 \\ \cline{3-8}
  & Consumption & 25-34 & 35-44 & 45-54 & 55-64 & 65-74 & 75+\\ \hline
  0 & Low & 0/106 & 5/169 & 21/159 & 34/173 & 36/124 & 8/39\\
  1 & High & 1/10 & 4/30 & 25/54 & 42/69 & 19/37 & 5/5\\ \hline
\multicolumn{6}{l}{{\em Source} : \cite{breslow-day1980}}
 \end{tabular}
 \label{table:Christensen}
\end{table}

The second example is from Table 4.16 in \cite{christensen1997} 
(data source is from \cite{breslow-day1980}).
Table \ref{table:Christensen} refers to the occurrence of esophageal
cancer in Frenchmen which were classified on ages and dummy variable on
alcohol consumption. 
We test the goodness-of-fit of the model (\ref{model:bivariate}) 
with $J=6$ and $K=2$ by likelihood ratio statistics $L_0$ via MCMC. 
Then the value of $L_0$ is $20.89$ and the asymptotic p-value is
$0.0003330$ 
from the asymptotic distribution $\chi^2_4$.  
We computed the exact distribution of $L_0$ via MCMC with 
${\cal B}_{\Gamma(A \otimes B)}$ and ${\cal B}_0^2$.
Figure \ref{fig:ch-a} represents the histograms of $L_0$.
The estimated p-values are $0.00011$ with 
${\cal B}_{\Gamma(A \otimes B)}$ and $0.00055$ with 
${\cal B}_0^2$. Therefore the model (\ref{model:bivariate}) is 
rejected at the significance level of 1\%.

\begin{figure}[htbp]
 \centering
 \begin{tabular}{cc}
 \includegraphics[scale=0.25]{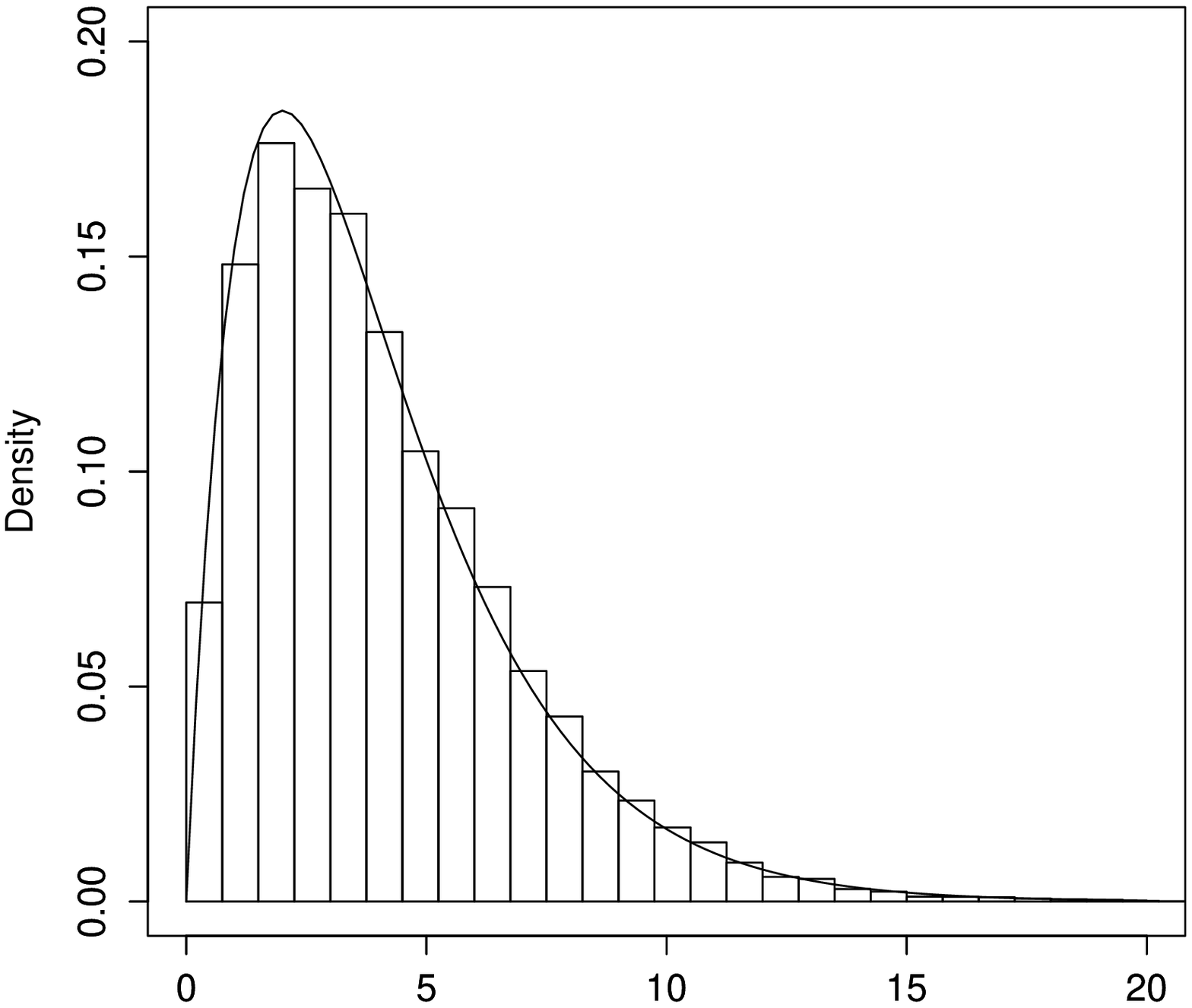} &
  \includegraphics[scale=0.25]{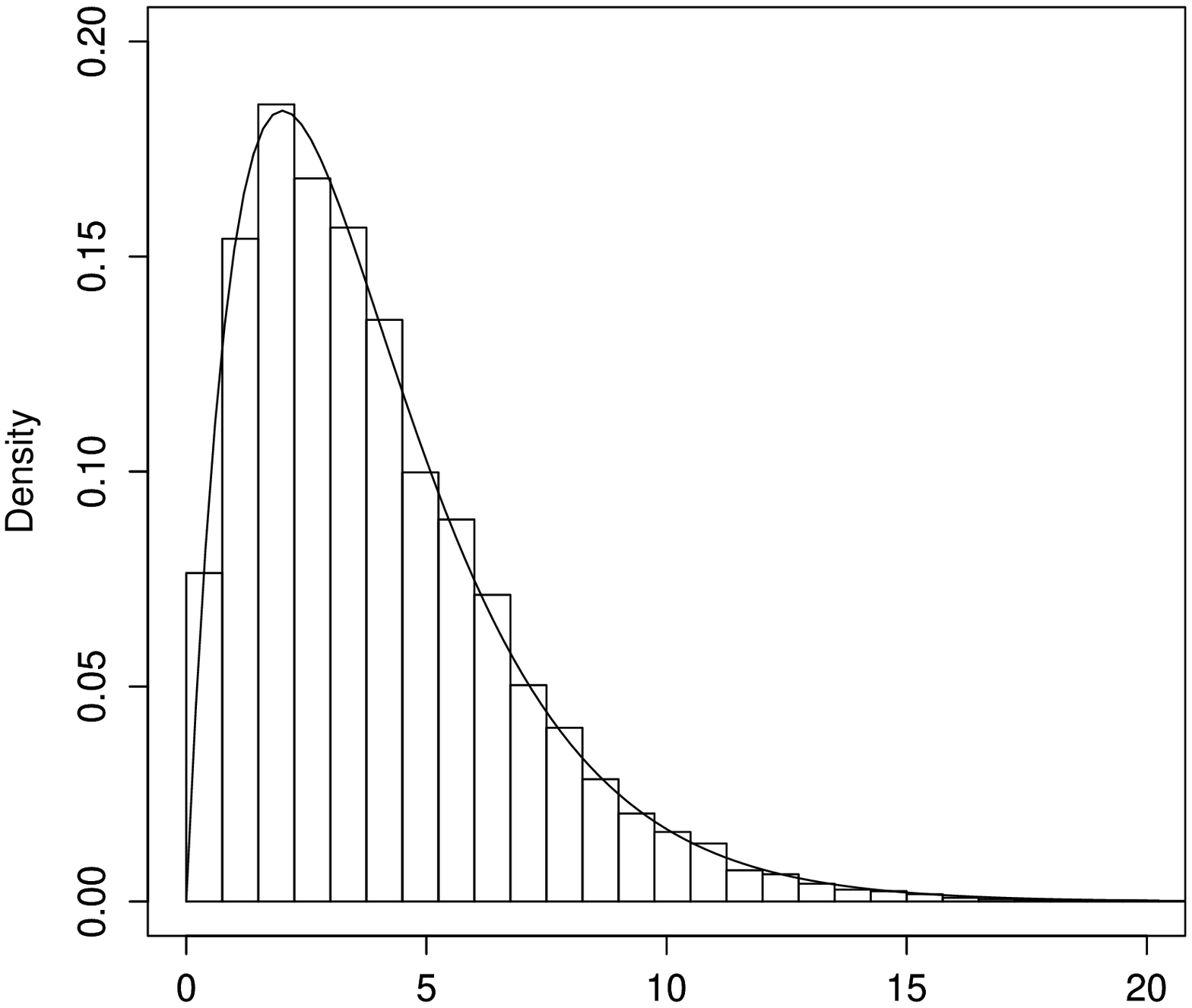}\\
  (a) a histogram with ${\cal B}_{\Lambda(A \otimes B)}$ & 
  (b) a histogram with ${\cal B}_0^2$
 \end{tabular}
 \caption{Histograms of $L_0$ via MCMC with ${\cal B}_{\Lambda(A \otimes
 B)}$ and ${\cal B}_0^2$}  
 \label{fig:ch-a}
\end{figure}

\section{Concluding remarks}
\label{sec:discussions}

In Theorem  \ref{theorem:bivariate}
we showed the connectivity result 
%for fibers with  positive response variable marginals 
for bivariate logistic regression.  
A natural extension of
Theorem  \ref{theorem:bivariate}
to $m$ covariates
is given as follows.
Let $\Bj=(j_1, \dots, j_m)$ 
denote the combination of $m$ levels and let $\bm{e}_{\Bj}$ denote an
array
with $1$ at the cell $(1,\Bj)$ and $-1$ at the cell $(2,\Bj)$.
Define ${\cal B}_{\Lambda(A_1 \otimes \dots \otimes
  A_m)}$  as  the set of the following moves $z$:
 \begin{enumerate}
  \item $z = \bm{e}_{\Bj_1} - \bm{e}_{\Bj_2} - \bm{e}_{\Bj_3},
        + \bm{e}_{\Bj_4}$
  \item $\Bj_1 - \Bj_2 = \Bj_3 - \Bj_4$ .
 \end{enumerate}
Then we conjecture the following.
\begin{conjecture}
The set of moves ${\cal B}_{\Lambda(A_1 \otimes \dots \otimes
  A_m)}$ connects every fiber with positive response marginals
for the logistic regression with $m$ covariates.
\end{conjecture}

The separation lemma and some 
steps of the proof of Theorem \ref{theorem:bivariate} can be easily
generalized to multiple logistic regression. However many steps of 
our proof, especially those for Cases 3 and 5, 
are restricted to the two-dimensional case.

As discussed in Section \ref{sec:examples} it seems that we
can further restrict to the set of moves  
$z = \bm{e}_{\Bj_1} - \bm{e}_{\Bj_2} - \bm{e}_{\Bj_3}    + \bm{e}_{\Bj_4}$,
where the elements of $\Bj_1 - \Bj_2 = \Bj_3 - \Bj_4$ are  $\pm 1$  or 0.
Hence a stronger conjecture (even for the case of $m=2$) is given as follows.

\begin{conjecture}
The subset of moves from ${\cal B}_{\Lambda(A_1 \otimes \dots \otimes
  A_m)}$ such that the elements of $\Bj_1 - \Bj_2 = \Bj_3 - \Bj_4$ are  $\pm 1$  or 0
connects every fiber with positive response marginals
for the logistic regression with $m$ covariates.
\end{conjecture}

In this paper we considered logistic regression, which is the Lawrence
lifting of Poisson regression.  Our Theorem \ref{thm:segre} describes 
Markov bases for a general Segre product of configurations. Therefore
it is interesting, in practice, to investigate connectivity result for Lawrence
lifting of a general Segre product of configurations.

In the bivariate logistic regression, it is interesting to test the
null hypothesis that the coefficient of one of the covariates is zero.
%% We can easily test $\beta=0$ in the bivariate Poisson or the bivariate
%% logistic regression.
Generating random samples under the null hypothesis is simple because
it reduces to univariate logistic regression as follows.
In (\ref{model:bivariate}) consider the null hypothesis 
$H_0: \beta=0$. Given  observed data $(x_{ijk})$, we can generate 
random sample from the null conditional distribution by MCMC procedure for
the marginals $(x_{1j+})$, $j=1,\dots,J$.  Then for each $j$,
we can sample $x_{1jk}$, $k=1,\dots,K$,
by the random sampling  without replacement.

\section*{Appendix}
\appendix
%\section{Proof of Proposition \ref{prop:1}}

%\section{Proof of Theorem \ref{theorem:univariate}}

\section{Proof of Theorem \ref{thm:segre}}

%\begin{proof}
Let $x$ and $y$ be two tables in the same fiber.
Write $z=x-y$.
First consider the case that $x$ and $y$ already have the same marginals:
\[
0=z_{j+}, \forall j,  \quad 0=z_{+k}, \forall k.
\]
Then, as is well known for two-way complete independence model,
we can use the basic moves
to move from $X$ to $Y$. Note that
(\ref{eq:move}) is always satisfied in these steps.

Next consider the case that the row sums are already the same
\[
z_{j+}=0, \ j=1,\dots,J, 
\]
but the column sums are not yet the same.
For the moment, ignoring joint frequencies, 
just look at the column sums of $x$ and $y$:
\[
(x_{+1}, \dots, x_{+K}), \quad 
(y_{+1}, \dots, y_{+K})
\]
We can use the moves of ${\cal B}_B$ to move from the marginal
frequency $(x_{+1}, \dots, x_{+K})$ to the marginal
frequency $(y_{+1}, \dots, y_{+K})$.   
However, of course we have to worry about the joint frequencies and
the row sums.  Here the idea is that we can ``distribute'' moves of ${\cal B}_B$ 
to the cells of the $J\times K$ table, in such a way that we never disturb
the row sums.  This way, we can make column sums equal, while always keeping the
row sums equal.
Consider a situation that a move $z^B$ of ${\cal B}_B$
% \[
% (z_1, \dots, z_J)=(z_1^+, \dots, z_J^+) - (z_1^-, \dots, z_J^-)
% \]
can be added to $(x_{+1}, \dots, x_{+K})$.  Then
we have
\[
x_{+k} \ge z^{B,-}_k,  \quad k=1,\dots,K.
\]
This shows that in each  column $k$ with $c_k=z^{B,-}_k > 0$, there are at least
$c_k$ positive frequencies of $x$, i.e., there exists  indices
$j_{1,k}, \dots, j_{c_k,k}$ such that 
\[
(x_{1k}, \dots, x_{Jk}) - (\Be_{j_{1,k}}+ \dots + \Be_{j_{c_k,k}}) \ge 0.
\]
Here ``$\ge 0$'' means that every component of the left-hand side is non-negative.
Collect the indices $j_{1,k}, \dots, j_{c_k,k}$ for
all $k$ with $z^{B,-}_k>0$  as $j_1, \dots, j_{\deg z^B}$.
Then $z^B(j_1, \dots, j_{\deg z^B})$ can be added  to $x$. Note that
$z^B$ is added to the marginal frequencies $(x_{+1}, \dots, x_{+K})$, but 
the move does not change the row sums of $x$.  
% For example let $J=4$ and consider the case that
% \[
% (z_1,z_2, z_3, z_4)=(-1,2,-1,0)
% \]
% is a move for $B$ and is applicable to $(x_{+1},\dots,x_{+4})$.
% Then there exists $i,i'$ such that $x_{i1}\ge 1$, $x_{i'3}\ge 1$.
% We can move these two frequencies ``horizontally'' to column 2.
% From my pictures it is clear that this kind of ``distributing frequencies'' is
% always possible.  
This argument implies that the set of moves $z^B(j_1, \dots, j_{\deg z^B})$
are sufficient for connecting two tables with the same row sums.

Lastly we consider the case that neither the row sums nor the column
sums are the same for $x$ and $y$.  We
can employ a ``greedy algorithm'', in which we first look at the row
sums only and try to make the row sums equal, because the column sums
can be adjusted later by the above argument.  Now in the above argument,
with the roles of the rows and the columns interchanged, 
we can ignore the fact that the column sums are not
yet equal.  We can use the same procedure as above.  Therefore,
by applying a move of the form
$z^A(k_1, \dots, k_{\deg z^A})$ we do 
not change the column sums of $x$ and $y$.  
% It is only essential that we can always subtract the negative part
% of a move $z^A(k_1, \dots, k_{\deg z^A})$. 
Then we can make the row sums of $x$ and $y$ equal, while not changing
the column sums of $x$ and $y$. 
%\end{proof}

\section{A separation lemma}
Here we prove a lemma, which is needed for our proof of 
Theorem \ref{theorem:bivariate}.

\begin{lemma}
\label{ref:separation}
Let ${\cal I}=[J]\times [K]$ and let $S^+$ and $S^-$ be disjoint
subsets of ${\cal I}$ satisfying the following properties:
\begin{enumerate}
\setlength{\itemsep}{0pt}
\item 
$(j,k)\in S^+, \ j'\le j, \ k'\le k \quad \Rightarrow \quad (j',k')\in
S^+. $
\item 
$(j,k)\in S^-, \ j'\ge j, \ k'\ge k \quad \Rightarrow \quad (j',k')\in
S^-. $
\item There are no distinct four points 
 $(j_1, k_1)\in S^+, 
 (j_2,  k_2)\in S^-, (j_3, k_3)\not\in S^+, (j_4, k_4)\in S^+$  
 and there are no distinct four points
$(j_1, k_1)\in S^-, 
(j_2,  k_2)\in S^+, (j_3, k_3)\not\in S^-, (j_4, k_4)\in S^-$  
such that 
\[
(j_1, k_1)- (j_2, k_2)=(j_3, k_3)-(j_4,k_4).
\]
\end{enumerate}
Then there exists a line with rational slope separating $S^+$
and $S^-$, i.e.\ there exist integers $a,b,c$, $((a,b)\neq (0,0))$,
such that
\begin{equation}
 \label{separation}
 S^+ \subset \{(j,k)\in {\cal I} \mid a j + b k \le c \}, \qquad 
 S^- \subset \{(j,k)\in {\cal I} \mid a j + b k \ge c \}.
\end{equation}
\end{lemma}

\begin{proof}
The lemma obviously holds if $S^+$ or $S^-$ is empty.  Therefore we only
need to consider case that $S^+$ and  $S^-$  are non-empty.
Define $j_l= \min\{ j \mid \exists (j,k) \in S^-\}$ and for $j\in
\{j_l, j_l+1, \dots, J\}$ define 
$f(j)= \min\{k \mid (j,k)\in S^- \}$. Let $f^*$ be the largest convex
minorant \citep{moriguti1953} of $f(j)$, $j\in \{j_l, j_l+1, \dots, J\}$, 
i.e.\ $f^*(\cdot)$ is the maximum among convex functions
not exceeding $f(j)$ for each $j\in \{j_l, j_l+1, \dots, J\}$. 
Then  $f^*$ consists of finite number of 
line segments. Let $j_1 < j_4$ be endpoints of a line segment and let
$L^*_{j_1, j_4}$ denote the line segment.  Then
$(j_1, f(j_1)), (j_4, f(j_4)) \in S^-$. Also by construction of $f^*$, 
\[
(j,k)\in S^-, \ j_1 \le j \le j_4 \quad \Rightarrow \quad  k\ge f^*(j).
\]
Therefore every point strictly below $L^*_{j_1, j_4}$ belongs to
$S^+$.
Consider the rectangular region  of integer points
\[
R_{j_1, j_4}= \{j_1, \dots, j_4\}\times \{f(j_4), \dots, f(j_1)\}.
\]
If there exists a point $(j_2,k_2)\in S^+ \cap R_{j_1, j_4}$ 
strictly  above the line segment $L^+_{j_1, j_4}$, let
$(j_3, k_3)=(j_1, k_1)- (j_2, k_2)+(j_4,k_4)\in (S^+)^C  \cap R_{j_1, j_4}$
and condition 3 of the lemma is violated.  This shows that 
no point of $R_{j_1, j_4}$ strictly above $L^+_{j_1, j_4}$ belongs to
$S^-$.  
Also the points above $R_{j_1, j_4}$ belong to $S^+$ by the
monotonicity condition (2).
Therefore $L^+_{j_1, j_4}$ is a separating line for
the interval $\{j_1, \dots, j_4\}$.

Now we similarly construct the smallest concave majorant $f_*(j)$ for $S^-$. Then
by a hyperplane separation theorem for two convex sets, there exists a
line with rational slope between $f_*(j)$ and $f^*(j)$. This prove the
lemma.
\end{proof}

\section{Proof of Theorem \ref{theorem:bivariate}}
  Let $x := \{x_{ijk}\}$ and $y := \{y_{ijk}\}$ be two
  $2 \times J \times K$ tables in the same fiber satisfying 
  $x_{+jk} = y_{+jk} > 0$. 
  Then $z := \{z_{ijk}\} = x - y$ is a move for $\Lambda(A \otimes B)$. 
  Let $z^1$ denote the $(i=1)$-slice of $z$. 
  As mentioned in Section \ref{sec:connectivity}, $z$ satisfies 
  $z_{+jk} = 0$, $z_{i++} = 0$, $\forall i,j,k$, and
  \begin{equation}
   \label{linear-constraint}
   \sum_{j=1}^J j z_{1j+} = 0 ,   \qquad 
   \sum_{k=1}^K k z_{1+k} = 0.
  \end{equation}
Note that  $z_{i++}=0$ implies 
\begin{equation}
\label{eq:j-direction}
\sum_{j=1}^J j z_{1j+} = 0  \quad \Leftrightarrow \quad
\sum_{j=1}^J (J-j+1) z_{1j+} = 0.
\end{equation}
Similarly $\sum_{k=1}^K k z_{1+k} = 0  \ \Leftrightarrow\ 
\sum_{k=1}^K (K-k+1) z_{1+k} = 0$.  This implies that when
we consider a sign pattern of a move, we can arbitrarily choose
directions for two factors $j$ and $k$.

Let ${\cal I}^+$ and ${\cal I}^-$ be the multisets of
indices defined by 
  \[
   {\cal I}^+ := \{(j,k) \mid z_{1jk} > 0\}, \quad 
   {\cal I}^- := \{(j,k) \mid z_{1jk} < 0\}, 
  \]
  where the multiplicity of $(j,k)$ in ${\cal I}^+$ and ${\cal I}^-$ is
  $\vert z_{1jk} \vert$.

  Suppose that $(j_1,k_1) \in {\cal I}^+$, $(j_2,k_2) \in {\cal I}^-$ 
  and $j_1 < j_2$. 
  Then we note that there exist $j_3 < j_4$, $k_3$ and $k_4$
  satisfying 
  \begin{equation}
   \label{condition1}
   (j_3,k_3) \in {\cal I}^- \setminus \{(j_2,k_2)\}, \quad 
   (j_4,k_4) \in {\cal I}^+ \setminus \{(j_1,k_1)\}
  \end{equation}
  from (\ref{linear-constraint}). 
  If $k_1 < k_2$ and $k_3 > k_4$, 
  there exists $k_5 < k_6$, $j_5$ and $j_6$
  satisfying 
  \[
   (j_5,k_5) \in {\cal I}^- \setminus \{(j_2,k_2), (j_3,k_3)\}, \quad 
   (j_6,k_6) \in {\cal I}^+ \setminus \{(j_1,k_1), (j_3,k_4)\}. 
  \]

Write 
$y(j_1,j_2;k_1,k_2)=\Be_{j_1 k_1}-\Be_{j_2 k_2}$ .
%and $z(j_3,j_4;k_3,k_4)$, respectively.
%% In the similar way we denote the patterns in Figure \ref{fig:Case2-1} by  
%% $-y(j_1,j_2;k_1,k_1)$ and $-y(j_3,j_4;k_3,k_3)$, respectively.
When a move $z$ includes $y(i_1,i_2;j_1,j_2)$, we denote it by 
$y(i_1,i_2;j_1,j_2) \subset z$. 

\bigskip 
\noindent {\bf Case 1.}\quad 
We first consider the case where there exist $j_0$, $j_1$, $j_2$, 
$k_0$, $k_1$ and $k_2$ such that 
\begin{equation}
 \label{condition:j}
  z_{1 j_0 k_1} > 0, \quad z_{1 j_0 k_2} < 0,
\end{equation}
\begin{equation}
 \label{condition:k}
  z_{1 j_1 k_0} > 0, \quad z_{1 j_2 k_0} < 0.
\end{equation}
%(we have to give a proof for this case)\\
Without loss of generality we assume  $j_1 < j_2$ and $k_1 < k_2$.
Let $S^+=\{(j,k) \mid \exists (j',k')\in {\cal I}^+, \ j\le j', k\le
k'\}$.
Similarly define $S^- =\{ (j,k) \mid \exists (j',k') \in 
{\cal I}^-, \ j\ge j', k\ge k'\}$.
We only need to consider the case that the condition
3 of Lemma \ref{ref:separation} is satisfied.  Also, 
if $S^+$ or  $S^-$ is not monotone in the sense of 
conditions 1 and 2 of Lemma \ref{ref:separation},  
we can reduce the $L_1$ distance between $x$ and $y$.  
This can be seen as follows. If 
$S^+$ or  $S^-$ is not monotone, then we can find a pattern
in Figure \ref{fig:Case2-1} (or a vertical pattern of this).
Without loss of generality let 
$j_2 - j_1 \le j_4 - j_3$
and define $j_5 := j_4 - (j_2-j_1)$.  For simplicity assume $j_2 <
j_3$ or $k_0\neq k_3$.
Then by applying 
\[
z_1 := - \bm{e}_{j_1 k_0} + \bm{e}_{j_2 k_0}
- \bm{e}_{j_5 k_1} + \bm{e}_{j_4 k_1}
\]
to $z$, we can reduce the $L_1$ distance between $x$ and $y$ by at least
four. 
%\hfill\qed\\
\begin{figure}[htbp]
 \[
 \begin{array}{cccccc}
  & j_1 & & & & j_2\\ \cline{2-6}
   k_0 & \multicolumn{1}{|c}{+} & 0 
   & \cdots & 0 & \multicolumn{1}{c|}{-}\\ 
  \cline{2-6}
 \end{array}\quad 
 \begin{array}{cccccccc}
  & j_3 & & j_5 & & & & j_4\\ \cline{2-8}
   k_1 & \multicolumn{1}{|c}{-} & 0 &  
   & \cdots & & 0 & \multicolumn{1}{c|}{+}\\ 
  \cline{2-8}
 \end{array}
 \]
 \caption{Case 1 or Case 2-1}
 \label{fig:Case2-1}
\end{figure}
The case of $k_0=k_1$ and $j_2=j_3$ needs a special consideration, but 
the monotonicity holds 
with respect to the horizontal separating line through $(k_0,j_2)$.
Therefore it suffices to
consider the case that $S^+$ and $S^-$ are monotone
in the sense of 
conditions 1 and 2 of Lemma \ref{ref:separation}. 
Then  \[
\sum_{j,k}(aj + bk + c)z_{1jk}=0
\]
implies that non-zero  elements $z_{1jk}\neq 0$ only exist on the
line $\{(j,k) \mid aj + bk + c\}$.  Then the problem reduces to the
univariate logistic regression.

\bigskip 
\noindent {\bf Case 2.}\quad 
Next we consider the case %where there do not exist  $j_0$, $k_1$ and
%$k_2$ satisfying
that only one of the patterns of 
 (\ref{condition:j}) or (\ref{condition:k}) exists.
Without loss of generality, we assume that (\ref{condition:k}) holds
 and 
from  Lemma \ref{lemma:appendix-1} we assume that 
there exist $j_1 < j_2 \le j_3 < j_4$
such that  
\[
 z_{1 j_1 +} > 0, \quad z_{1 j_2 +} < 0, \quad z_{1 j_3 +} < 0, \quad
 z_{1 j_4 +} > 0.
\]
%% When 
%% %there exists $k_0$ satisfying (\ref{condition:k}), 
%% (\ref{condition:k}) holds,, 
In this case
either a pattern of signs in Figure \ref{fig:Case2-1} 
or a pattern in Figure \ref{fig:Case2-2} has to exist in $z^1$.

%% Denote the left pattern %(i) and (ii) 
%% in Figure \ref{fig:Case2-3} 
%% by

\begin{figure}[htbp]
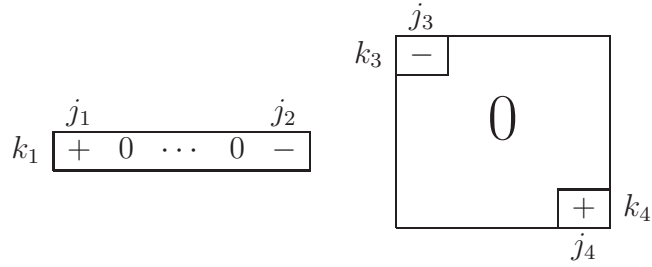

 \[
 \begin{array}{c}
 \begin{array}{cccccc}
  & j_1 & & & & j_2\\ \cline{2-6}
   k_1 & \multicolumn{1}{|c}{+} & 0 
   & \cdots & 0 & \multicolumn{1}{c|}{-}\\ 
  \cline{2-6}
 \end{array}\quad
 \begin{array}{ccccccc}
  & j_3 \\ \cline{2-6}
   k_3 & \multicolumn{1}{|c|}{-} & & & & \multicolumn{1}{c|}{}\\ 
  \cline{2-2}
   & \multicolumn{1}{|c}{} & & & & \multicolumn{1}{c|}{}\\ 
  & \multicolumn{1}{|c}{} & & \hsymb{0} & & \multicolumn{1}{c|}{}\\ 
  & \multicolumn{1}{|c}{} & &  & & \multicolumn{1}{c|}{}\\ 
  \cline{6-6}
   & \multicolumn{1}{|c}{} & & & & \multicolumn{1}{|c|}{+}
   & k_4 \\ \cline{2-6}
   & & & & & j_4
 \end{array} %\\
%%   \text{(i)}\\
%%  \begin{array}{cccccc}
%%   & j_1 & & & & j_2\\ \cline{2-6}
%%    k_1 & \multicolumn{1}{|c}{+} & 0 
%%    & \cdots & 0 & \multicolumn{1}{c|}{-}\\ 
%%   \cline{2-6}
%%  \end{array}\quad
%%  \begin{array}{ccccccc}
%%   & & & & & j_3 \\ \cline{2-6}
%%    & \multicolumn{1}{|c}{} & & & & \multicolumn{1}{|c|}{+} & k_3\\ 
%%   \cline{6-6}
%%    & \multicolumn{1}{|c}{} & & & & \multicolumn{1}{c|}{}\\ 
%%   & \multicolumn{1}{|c}{} & & \hsymb{0} & & \multicolumn{1}{c|}{}\\ 
%%   & \multicolumn{1}{|c}{} & &  & & \multicolumn{1}{c|}{}\\ 
%%   \cline{2-2}
%%    k_4 & \multicolumn{1}{|c|}{-} & & & & \multicolumn{1}{c|}{}
%%    & \\ \cline{2-6}
%%    & j_4 & & & & 
%%  \end{array} \\
%%   \text{(ii)}  
\end{array}
\]
 \caption{Case 2-2}
 \label{fig:Case2-2}
\end{figure}

%% \subsection{Proof of Case 2-1 in Figure \ref{fig:Case2-1}}
\medskip\noindent
{\bf Case 2-1.} \ The case of Figure \ref{fig:Case2-1}.\\
% \quad $j_2 - j_1 \le j_4 - j_3$.\\
In this case we can reduce the $L_1$ distance between $x$ and $y$ as in Case 1.

%\subsection{Proof of Case 2-2 in Figure \ref{fig:Case2-2}}

\medskip\noindent
{\bf Case 2-2.} \ The case of Figure \ref{fig:Case2-2}.\\
In the case of  Figure \ref{fig:Case2-2},    
we  distinguish two subcases depending on 
$j_2- j_1 \le j_4 - j_3$ or 
$j_2- j_1 > j_4 - j_3$.
\begin{description}
  \item[Case 2-2-1. ] $j_2 - j_1 \le j_3 - j_4$. \\
  %Suppose that $j_2 - j_1 < j_3 - j_4$ and 
  Let $j_5:= j_3 + (j_2 - j_1)$. 
             By applying 
             \[
             z_{2a} := 
             - \bm{e}_{j_1 k_1} + \bm{e}_{j_2 k_1}
             + \bm{e}_{j_3 k_3} - \bm{e}_{j_5 k_3} 
             \]
             to $z$, we reduce the $L_1$ distance by four.
%             \hfill \qed
\end{description}

\begin{description}         
\item[Case 2-2-2.]  $j_2 - j_1 > j_4 - j_3$.  \\
In this case we prove the theorem by induction on $j_4-j_3$. 
When $j_4-j_3=0$, the problem is reduced to Case 1.  
Therefore we assume that $j_4 - j_3 > 0$.
\end{description}

\begin{description}         
 \item[Case 2-2-2-1.] 
	    $(x_{1, j_1+1,k_1}, x_{1, j_4-1,k_4}) > 0$ or 
	    $(x_{2, j_1+1,k_1}, x_{2, j_4-1,k_4}) > 0$. \\
	    In this case we can apply
            \[
            z_{2b} :=
            - \bm{e}_{j_1 k_1} + \bm{e}_{j_1+1, k_1}
            + \bm{e}_{j_4-1, k_4} - \bm{e}_{j_4, k_4}
            \]
            to $z$ and then 
            \[
            y(j_3,j_4-1;k_3,k_4) \subset 
            z + z_{2b}, \quad 
            \Vert z + z_{2b} \Vert_1 = \Vert z \Vert_1.
            \]
            Hence 
            $\Vert z \Vert_1$ is can be reduced by moves of 
            ${\cal B}_{\Lambda(A \otimes B)}$ from the inductive assumption. 
            
            In the case where 
            $(x_{1, j_2-1,k_2}, x_{1, j_3+1,k_3}) > 0$ or 
            $(x_{2, j_2-1,k_2}, x_{2, j_3+1,k_3}) > 0$, 
            the proof is similar.
%            \hfill \qed
 \item[Case 2-2-2-2. ] $(x_{1, j_1+1,k_1},x_{2, j_3+1, k_3})>0$ 
            or
            $(x_{2, j_1+1,k_1},x_{1, j_3+1, k_3})>0$, \\
            In this case we can apply 
            \[
            z_{2c} :=
            - \bm{e}_{j_1 k_1} + \bm{e}_{j_1+1, k_1}
            + \bm{e}_{j_3 k_3} - \bm{e}_{j_3+1, k_3}
            \]
            to $z$ and then 
            \[
            y(j_3+1,j_4;k_3,k_4) \subset 
            z + z_{2c} .
            \]
            Therefore 
            $\Vert z \Vert_1$ can be reduced by moves of 
            ${\cal B}_{\Lambda(A \otimes B)}$ from the inductive
            assumption.  

            In the case where 
            $(x_{1 j_2-1,k_2}, x_{2 j_4-1,k_4}) > 0$ or  
            $(x_{2 j_2-1,k_2}, x_{1 j_4-1,k_4}) > 0$, 
            the proof is similar.
%            \hfill \qed

 \item[Case 2-2-2-3. ] 
            $(x_{1, j_1+1, k_1},x_{2, j_2-1,k_2}, x_{1, j_3+1,k_3},
            x_{2, j_4-1,k_4}) = 0.$\\
            In this case we have
            \[
            (x_{2, j_1+1, k_1},x_{1, j_2-1,k_2}, x_{2, j_3+1,k_3},
            x_{1, j_4-1,k_4}) > 0.
            \]
            Then there exists $j_1 < j_5 < j_2$ such that 
            $(x_{2 j_5 k_1},x_{1 j_5+1,k_1}) > 0$. 
            Hence we can apply 
            \[
            z_{2d}^1 := - \bm{e}_{j_1k_1} 
            + \bm{e}_{j_1,k_1+1} 
            + \bm{e}_{j_5 k_1} 
            - \bm{e}_{j_5+1,k_1} 
            \]
            and 
            \[
            z_{2d}^2 := - \bm{e}_{j_5 k_1} 
            + \bm{e}_{j_5+1,k_1} 
            + \bm{e}_{j_4-1, k_4} 
            - \bm{e}_{j_4 k_4} 
            \]
            to $z$ in this order and then we have 
            \[
            \Vert z + z_{2d}^1 + z_{2d}^2 \Vert_1 
            = \Vert z \Vert_1 \quad \text{and} \quad
            y(j_1,k_1+1;j_2,k_2) \subset
            z + z_{2d}^1 + z_{2d}^2.
            \]
            Hence theorem holds from the inductive assumption.

            In the case where 
            $(x_{2, j_1+1, k_1},x_{1, j_2-1,k_2}, x_{2, j_3+1,k_3},
            x_{1, j_4-1,k_4}) > 0$, 
            the proof is similar.
%            \hfill \qed
\end{description}

%\subsection{Proof of Case 3 in Figure \ref{fig:Case3}}
\bigskip

\begin{figure}[htbp]
\[
 \begin{array}{cc}
\begin{array}{ccccccc}
 & j_1 \\ \cline{2-6}
  k_1 & \multicolumn{1}{|c|}{+} & & & &
  \multicolumn{1}{c|}{}\\ 
 \cline{2-2}
  & \multicolumn{1}{|c}{} & & \hsymb{0} & & \multicolumn{1}{c|}{}\\ 
 \cline{6-6}
  & \multicolumn{1}{|c}{} & & & & \multicolumn{1}{|c|}{-}
  & k_2 \\ \cline{2-6}
  & & & & & j_2
\end{array} &
\begin{array}{ccccc}
 & j_3 \\ \cline{2-4}
  k_3 & \multicolumn{1}{|c|}{-}  & &
  \multicolumn{1}{c|}{}\\ 
 \cline{2-2}
  & \multicolumn{1}{|c}{} & &  \multicolumn{1}{c|}{}\\ 
 & \multicolumn{1}{|c}{} & \hsymb{0} & \multicolumn{1}{c|}{}\\ 
 & \multicolumn{1}{|c}{} & & \multicolumn{1}{c|}{}\\ 
 \cline{4-4}
  & \multicolumn{1}{|c}{} & & \multicolumn{1}{|c|}{+}
  & k_4 \\ \cline{2-4}
  & & & j_4
\end{array}\\
%\text{(i)} & \text{(ii)}
\end{array}
\]
 \caption{Case 3}
 \label{fig:Case3}
\end{figure}

\begin{figure}[htbp]
\[
 \begin{array}{c}
  \begin{array}{ccccccc}
   & j_1 \\ \cline{2-6}
    k_1 & \multicolumn{1}{|c|}{+} & & & &
    \multicolumn{1}{c|}{}\\ 
   \cline{2-2}
    & \multicolumn{1}{|c}{} & & \hsymb{0} & & \multicolumn{1}{c|}{}\\ 
   \cline{6-6}
    & \multicolumn{1}{|c}{} & & & & \multicolumn{1}{|c|}{-}
    & k_2 \\ \cline{2-6}
    & & & & & j_2
  \end{array}\quad 
  \begin{array}{ccccc}
   & & & j_4 \\ \cline{2-4}
    & \multicolumn{1}{|c}{}  & &
    \multicolumn{1}{|c|}{+} & k_4\\ 
   \cline{4-4}
    & \multicolumn{1}{|c}{} & &  \multicolumn{1}{c|}{}\\ 
   & \multicolumn{1}{|c}{} & \hsymb{0} & \multicolumn{1}{c|}{}\\ 
   & \multicolumn{1}{|c}{} & & \multicolumn{1}{c|}{}\\ 
   \cline{2-2}
    k_3 & \multicolumn{1}{|c|}{-} & & \multicolumn{1}{c|}{}
    & \\ \cline{2-4}
    & j_3 & & 
  \end{array}\\
%    \text{(i)}\\
%   \begin{array}{ccccccc}
%    & j_5 \\ \cline{2-6}
%     k_5 & \multicolumn{1}{|c|}{-} & & & &
%     \multicolumn{1}{c|}{}\\ 
%    \cline{2-2}
%     & \multicolumn{1}{|c}{} & & \hsymb{0} & & \multicolumn{1}{c|}{}\\ 
%    \cline{6-6}
%     & \multicolumn{1}{|c}{} & & & & \multicolumn{1}{|c|}{+}
%     & k_6 \\ \cline{2-6}
%     & & & & & j_6
%   \end{array}\quad 
%   \begin{array}{ccccc}
%    & & & j_6 \\ \cline{2-4}
%     & \multicolumn{1}{|c}{}  & &
%     \multicolumn{1}{|c|}{-} & k_6\\ 
%    \cline{4-4}
%     & \multicolumn{1}{|c}{} & &  \multicolumn{1}{c|}{}\\ 
%    & \multicolumn{1}{|c}{} & \hsymb{0} & \multicolumn{1}{c|}{}\\ 
%    & \multicolumn{1}{|c}{} & & \multicolumn{1}{c|}{}\\ 
%    \cline{2-2}
%     k_5 & \multicolumn{1}{|c|}{+} & & \multicolumn{1}{c|}{}
%     & \\ \cline{2-4}
%     & j_5 & & 
%   \end{array} \\
%   \text{(ii)}\\
  \end{array}
\]
 \caption{Case $5$}
 \label{fig:Case3a}
\end{figure}

\noindent {\bf Case 3.}\quad 
We now consider the case that there exist no $j_0, k_1, k_2$
satisfying (\ref{condition:j}) and there exist no $k_0, j_1, j_2$
satisfying (\ref{condition:k}). 
From Lemma \ref{lemma:appendix-1}, 
either of the patterns of signs as in Figure \ref{fig:Case3} 
and Figure \ref{fig:Case3a} has to exist in $z^1$.  
Here we consider the case that patterns in Figure \ref{fig:Case3} exist.
% \ref{fig:Case3a}(ii). 
% When there exists no $k_0$ satisfying (\ref{condition:k}) and  
% \[
%   z_{1 j_3 k_0} < 0, \quad z_{1 j_2 k_0} > 0.
% \]
% either of the patterns of signs as in Figure \ref{fig:Case3} 
% and Figure \ref{fig:Case3a}(i) has to exist in $z^1$.  
% We can easily see that if there exists a pattern of Figure
% \ref{fig:Case3a}(i), there also exists a pattern in Figure
% \ref{fig:Case3a}(ii). 
%Hence the case of Figure \ref{fig:Case3a} reduces to the case of Figure 
%\ref{fig:Case3}.
%Therefore it suffices to consider the case of Figure \ref{fig:Case3}.
%% Therefore it suffices to consider the cases where $z^1$ contains
%% either pattern in Figure \ref{fig:Case2-1}, \ref{fig:Case2-2}(i) or 
%% \ref{fig:Case2-3}. 
%In the rest of this section we give proofs for each pattern.
The case of Figure  \ref{fig:Case3a}
will be treated as Case 5 below.
We make various subcases depending on the sizes of two rectangles in
Figure \ref{fig:Case3}. 

\begin{description}
 \item[Case 3-1. ]
            $j_2-j_1 \ge j_4-j_3$ and $k_2-k_1 \ge k_4-k_3$.\\
            In this case the left rectangle contains the right 
            rectangle in Figure \ref{fig:Case3}.
            Define $(j_5,k_5)$ by 
            \[
            (j_5,k_5) := (j_1,k_1) -  (j_3,k_3) +  (j_4,k_4) .
            \]
            Then 
            \[
             z_{3a} := - \bm{e}_{j_1 k_1} + \bm{e}_{j_5 k_5}
            + \bm{e}_{j_3 k_3} - \bm{e}_{j_4 k_4}
            \]
            reduces the $L_1$ distance by four.

            In the case where 
            $j_2-j_1 \le j_4-j_3$ and $k_2-k_1 \le k_4-k_3$, 
            the proof is similar. 
\end{description}

\begin{description}
 \item[Case 3-2. ] 
 When there is no inclusion relation between two rectangles of Figure
 \ref{fig:Case3}, it suffices to consider the case of $j_2-j_1 > j_4-j_3$ and 
 $k_2-k_1 < k_4-k_3$. 
 We prove the theorem by induction on $l := (k_2-k_1) + (j_4 - j_3)$.
 If $l = 0$, the theorem holds from Case 1.
\end{description}

\begin{description}
 \item[Case 3-2-1. ] $(x_{1j_1,k_1+1},x_{1j_4,k_4-1}) > 0$.\\ 
            In this case we can apply 
            \[
            z_{3b} := -\bm{e}_{j_1 k_1} + \bm{e}_{j_1, k_1+1}
            +\bm{e}_{j_4, k_4-1} + \bm{e}_{j_4 k_4}
            \]
            and then we have
            \[
            y(j_1,j_2;k_1+1,k_2) \subset z + z_{3b}, \quad 
            y(j_3,j_4;k_4,k_4-1) \subset z + z_{3b}.
            \]
            From the inductive assumption 
            the theorem holds in this case.
                        
            Also in the following cases, the proof is similar.
 \begin{itemize}
  \item $(x_{2j_1,k_1+1},x_{2j_4,k_4-1}) > 0$ ;
  \item $(x_{i,j_1+1,k_1},x_{i,j_4-1,k_4}) > 0$ ;
  \item $(x_{ij_2,k_2-1},x_{ij_3,k_3+1}) > 0$ ;
  \item $(x_{i,j_2-1,k_2},x_{i,j_3+1,k_3}) > 0$ ;
 \end{itemize}

 \item[Case 3-2-2. ]  $(x_{1j_1,k_1+1},x_{2j_3,k_3+1}) > 0$. \\
            In this case we can apply 
            \[
            z_{3c} := -\bm{e}_{j_1 k_1} + \bm{e}_{j_1, k_1+1}
            +\bm{e}_{j_3 k_3} + \bm{e}_{j_3, k_3+1}
            \]
            and then we have 
            \[
            y(j_1, j_2 ; k_1 + 1, k_2) \subset z + z_{3c}, 
            \quad 
            y(j_3, j_4 ; k_3 + 1, k_4) \subset z + z_{3c}.
            \]
            From the inductive assumption 
            the theorem holds in this case.
                        
            Also in the following cases, the proof is in the similar
            way.  
            \begin{itemize}
             \item $(x_{2,j_1,k_1+1},x_{1,j_3,k_3+1}) > 0$ ;
             \item $(x_{i,j_1+1,k_1},x_{i^*,j_3+1,k_3}) > 0$ ;
             \item $(x_{i,j_2,k_2-1},x_{i^*,j_4,k_4-1}) > 0$ ;
             \item $(x_{i,j_2-1,k_2},x_{i^*,j_4-1,k_4}) > 0$, 
            \end{itemize}
            where $i^* := 3 - i$.
%            \hfill\qed
\end{description}                      

\begin{description}
\item[Case 3-2-3.]   $x_{2j_1,k_1+1} = x_{1j_2,k_2-1}
  = x_{2j_3,k_3+1} = x_{1j_4,k_4-1}   =0$.\\
From the result of Case 3-2-1  and Case 3-2-2, 
it suffices to consider the case where 
\begin{equation}
 \label{case:3}
  x_{2j_1,k_1+1} = x_{1j_2,k_2-1}
  = x_{2j_3,k_3+1} = x_{1j_4,k_4-1} 
  =0. 
\end{equation}
We note that (\ref{case:3}) implies
\[
(x_{1 j_1,k_1+1}, x_{2 j_2,k_2-1}, 
x_{1 j_3,k_3+1}, x_{2 j_4,k_4-1})
> 0.
\]
Therefore there exist $j_1 < j_5 < j_2$ and $k_1 < k_5 < k_2$ satisfying
either 
\begin{equation}
 \label{pattern1}
  x_{1 j_5 k_5} > 0 \quad x_{2 j_5,k_5+1} > 0 
\end{equation}
or
\begin{equation}
 \label{pattern2}
  x_{1 j_5 k_5} > 0 \quad x_{2,j_5+1,k_5} > 0.
\end{equation} 
\end{description}

\begin{description}
 \item[Case 3-2-3-1. ] $x_{1 j_5 k_5} > 0$ and $x_{2 j_5,k_5+1} > 0 $ \quad
(\ref{pattern1}).\\
%            Suppose that (\ref{pattern1}) is satisfied.
            In this case we can apply 
            \[
            z_{3d}^1 := - \bm{e}_{j_1 k_1} + \bm{e}_{j_1,
            k_1+1} + \bm{e}_{j_5 k_5} - \bm{e}_{j_5, k_5+1}
            \]
            and 
            \[
            z_{3d}^2 := - \bm{e}_{j_5 k_5} + \bm{e}_{j_5,k_5+1} 
            + \bm{e}_{j_4 k_4-1} - \bm{e}_{j_4 k_4}
            \]
           to $z$ in this order. 
           Then we have 
           $\Vert z + z_{3d}^1 + z_{3d}^2 \Vert_1 
           = \Vert z \Vert_1$
           and 
           \[
           y(j_1,j_2;k_1+1,k_2) \subset 
           z + z_{3d}^1 + z_{3d}^2, 
           \quad
           y(j_3,j_4;k_3,k_4-1) \subset
           z + z_{3d}^1 + z_{3d}^2 .
           \]
           Hence from the inductive assumption 
            the $L_1$ distance can be reduced by moves in 
            ${\cal B}_{\Lambda(A \otimes B)}$. 
%            \hfill \qed

\end{description}

%We can assume that (\ref{pattern1}) is not satisfied.

\begin{description}
\item[Case 3-2-3-2.]
  $x_{1 j_5 k_5} > 0$ and $x_{2,j_5+1,k_5} > 0$ \quad
 (\ref{pattern2}).\\
In this case we further consider subcases depending on the
value of $x_{1,j_1+1,k}$.
\end{description}

\begin{description}  
\item[Case 3-2-3-2-1.] $x_{1,j_1+1,k_1} > 0$.\\ 
From the result of Case 3-2-1 and and Case 3-2-2,
it suffices to consider the case where 
            \begin{equation}
             \label{case:3-2}
              x_{2,j_1+1,k_1} = x_{1,j_4-1,k_4} 
              = x_{2,j_3+1,k_3} = x_{1,j_2-1,k_2}=0. 
            \end{equation}
            We note that (\ref{case:3-2}) implies that 
            \[
            (
            x_{1,j_1+1,k_1}, x_{2,j_4-1,k_4}, 
            x_{1,j_3+1,k_3}, x_{2,j_2-1,k_2}
            ) > 0.
            \]
            Since (\ref{pattern2}) is satisfied, 
            we can apply 
            \[
            z_{3e}^1 := - \bm{e}_{j_1 k_1} + \bm{e}_{j_1+1,k_1}
            + \bm{e}_{j_5 k_5} - \bm{e}_{j_5+1, k_5} 
            \]
            and 
            \[
            z_{3e}^2 := - \bm{e}_{j_5 k_5} + \bm{e}_{j_5+1,k_5} 
            + \bm{e}_{j_4-1, k_4} - \bm{e}_{j_4 k_4}
            \]
            in this order. Then we have                          
            $\Vert z + z_{3e}^1 + z_{3e}^2 \Vert_1 =
            \Vert z \Vert_1$ 
            and 
            \[
             y(j_1+1,j_2;k_1,k_2) \subset
            z + z_{3e}^1 + z_{3e}^2, \quad
            y(j_3,j_4-1;k_3,k_4) \subset
            z + z_{3e}^1 + z_{3e}^2. 
            \]
            Hence from the inductive assumption, 
            $L_1$ distance can be reduced by moves in 
            ${\cal B}_{\Lambda(A \otimes B)}$. 

If any of $x_{2,j_2-1,k_2}, x_{1,j_3+1,k_3}, x_{2,j_4-1,k_4}$ is
positive, the same argument can be applied.
\end{description}

\begin{description}
\item[Case 3-2-3-2-2. (Case 4)] \ 
$
  x_{1,j_1+1,k_1} = x_{2,j_2-1,k_2} 
  = x_{1,j_3+1,k_3} = x_{2,j_4-1,k_4} 
  = 0
$.\\
For readability, we relabel this case as Case 4.
In this case
\[
(
x_{2,j_1+1,k_1}, x_{1,j_2-1,k_2}, 
x_{2,j_3+1,k_3}, x_{1,j_4-1,k_4}
) > 0.
\]
Then there exists $j_1 < j_6 < j_2$ and 
$k_1 < k_6 < k_2$ satisfying either  
\begin{equation}
 \label{pattern3}
  x_{2j_6 k_6} > 0, \quad x_{1,j_6+1,k_6} > 0 
\end{equation}
or
\begin{equation}
 \label{pattern4}
  x_{2j_6 k_6} > 0, \quad x_{1,j_6,k_6+1} > 0.
\end{equation} 
\end{description}

\begin{description}
 \item[Case 4-1. ]
            The case that (\ref{pattern3}) is satisfied. \\
            In this case the proof is in similar to Case 3-2-3-1.
%2-3d.  
\end{description}

\begin{description}
 \item[Case 4-2. ]
The case that (\ref{pattern3}) is not satisfied.\\
In this case 
\begin{equation}
 \label{pattern5}
  x_{1 j k_1} = 0, \quad j = j_1+1, \ldots, j_2, \quad 
  x_{2 j_3 k} = 0, \quad k = k_3+1, \ldots, k_4.
\end{equation}
% Suppose that  $j_1 < j_3 < j_2$ and $k_3 < k_1 < k_4$.
% However this contradicts to (\ref{pattern5}).
%Hence we can assume without loss of generality
We can assume without loss of generality that $j_2 < j_3$. 
Then we note that 
\[
(j_7,k_7) := (j_4,k_4) - (j_2,k_2) + (j_1,k_1) \in {\cal J}
\]
where %${\cal J} := \{(j,k) \mid 1 \le j \le J, 1 \le k \le K\}$.
${\cal J} := [J] \times [K]$. 
\end{description}

\begin{description}
 \item[Case 4-2-1. ] $x_{2 j_7 k_7} > 0$ or $y_{1 j_7 k_7} > 0$.\\
            In this case we can apply 
            \[
            z_{4a} := -\bm{e}_{j_1 k_1} + \bm{e}_{j_2 k_2}
            +\bm{e}_{j_7 k_7} - \bm{e}_{j_4 k_4}
            \]
            to $z$  and 
            we can reduce the $L_1$ distance by four.
\end{description}

\begin{description}
 \item[Case 4-2-2. ] 
$x_{2 j_7 k_7} = 0$ and $y_{1 j_7 k_7} = 0$.\\
In this case  we have $z_{1 j_7 k_7} > 0$.
\end{description}

\begin{description}
 \item[Case 4-2-2-1.] 
The case that there exists $j_7 < j_8 < j_3$ such that 
            $z_{1 j_8 k_7} < 0$.\\
In this case we can prove the theorem in the same way as  Case 2-2.
\end{description}

\begin{description}
 \item[Case 4-2-2-2.] 
The case that  $z_{1 j k_7} \ge 0$ for all $j_7 < j < j_3$.\\
From the condition (\ref{pattern5}) there exists $j_9$ satisfying either
of the following conditions, 
\begin{enumerate}
 \item[(i)] $j_7 \le j_9 < j_3$, 
            $z_{1 j_9 k_7} > 0$ and 
            $x_{1, j_9+1, k_7} > 0$ ; 
 \item[(ii)] $j_7 < j_9 < j_3$, 
            $z_{1 j_9 k_7} = z_{1 j_9+1, k_7} = 0$, 
            $x_{2 j_9 k_7} > 0$
            and $x_{1 j_9+1, k_7} > 0$.
\end{enumerate}
\end{description}

\begin{description}
 \item[Case 4-2-2-2-1. ]
            The case that  (i) is satisfied. \\
            In this case by applying the move 
            \[
            z_{4b} := -\bm{e}_{j_9 k_7} + \bm{e}_{j_9+1, k_7}
            + \bm{e}_{j_4 -1, k_4} - \bm{e}_{j_4 k_4}, 
            \]
            we have 
            $\Vert z + z_{4b} \Vert_1 = \Vert z \Vert_1$
            and  
            \[
            y(j_1,j_2;k_1,k_2) \subset z + z_{4b}, \quad
            y(j_3,j_4-1;k_3,k_4) \subset z + z_{4b}.
            \]
            Hence the theorem holds from the inductive assumption.
 \item[Case 4-2-2-2-2. ] The case that (ii) is satisfied. \\
            In this case  by applying the move 
            $z_{4b}$ and 
            \[
            z_{4c} := -\bm{e}_{j_1 k_1} + \bm{e}_{j_1+1, k_1}
            + \bm{e}_{j_9 k_7} + \bm{e}_{j_9+1, k_7}
            \]
            in this order and then we have 
            $\Vert z + z_{4b} + z_{4c} \Vert_1 
            = \Vert z \Vert_1$ and
            \[
            y(j_1+1,j_2;k_1,k_2) \subset z + z_{4b} +
            z_{4c},\quad 
            y(j_3,j_4-1;k_3,k_4) \subset z + z_{4b} + z_{4c} 
            \]
            Hence the theorem holds from the inductive assumption.
\end{description}

\noindent {\bf Case 5.}\quad 
We now consider the case where $z^1$ contains patterns of signs  in
Figure \ref{fig:Case3a} and does not contain patterns of signs  in
Figure \ref{fig:Case3}.  
We show that if $z_1$ contains the pattern of signs in Figure
\ref{fig:Case5-1}, we can reduce the $L_1$ norm  $z$ or otherwise $z$
is not a move. 
The proof is by induction on 
\[
l := \min \left((j_2 - j_1) + (k_2 - k_1), (j_4 - j_3) + (k_3 - k_4)
 	    \right).
\]
When $l=1$, theorem holds by Case 1.
\begin{description}
 \item[Case 5-1. ] $k_2 > k_4$\\
	    Based on the argument in Case 3-2-1 and 3-2-2, we only need to consider
	    the case that $z^1$ contains patterns in Figure \ref{fig:Case5-1}, where
	    $z_{1jk}=0^*$ and $z_{1jk}=0_*$ denote $x_{1jk} = y_{1jk} > 0$ and 
	    $x_{2jk} = y_{2jk} > 0$, respectively. 
	    Define two set of cells ${\cal A}_1$ and ${\cal A}_2$ as in Figure
	    \ref{fig:Case5-1}. Then there exist $j_5$, $k_5$, $j_6$ and
	    $k_6$ such 
	    that  
	    \[
	    j_1 < j_5 < j_2, \quad j_3 < j_6 < j_4, \quad 
	    k_1 < k_5 < k_2, \quad k_4 < k_6 < k_4, 
	    \]
	    \[
	    z_{1 j_5 k_5} = 0^*, \quad z_{1 j_5 k_5 + 1} = 0_*, \quad 
	    z_{1 j_6 k_6} = 0_*, \quad z_{1 j_6, k_6+1} = 0^*
	    \]
	    as represented in Figure \ref{fig:Case5-1}(i). 
	    Then we can apply the move
	    \[
	    z_{5a} := \bm{e}_{j_5 k_5} - \bm{e}_{j_5, k_5+1} - \bm{e}_{j_6 k_6} -
	    \bm{e}_{j_5, k_6 + 1}, 
	    \]
	    \[
	    z_{5b} := \bm{e}_{j_6 k_6} - \bm{e}_{j_6, k_6+1}
	    - \bm{e}_{j_4 k_4} + \bm{e}_{j_4,k_4+1} 
	    \]
	    to $z$ in this order and $z' := z + z_{5a} + z_{5b}$ is expressed as in Figure
	    \ref{fig:Case5-1}(ii).  
	    Suppose that there exists $(j,k) \in {\cal A}_1$ such that
	    $z_{1jk} < 0$.   
	    Then $z_{1 j k_4} =0$ and hence there exists 
	    $k \le k' < k_4$ such that 
	    \[
	     z_{1 j k'} < 0, \quad  z_{1 j, k'+1} = 0. 
	    \]
	    Therefore we can apply the move 
	    \[
     	    z_{5c} := -\bm{e}_{j_5 k_5} + \bm{e}_{j_5, k_5+1}
     	    + \bm{e}_{j k'} - \bm{e}_{j,k'+1}
	    \]
	    and $z'' := z' + z_{5c}$ satisfies 
	    $$
	    \Vert z'' \Vert_1 = \Vert z \Vert_1, \quad 
	    y(j_1,j_2;k_1,k_2) \subset z''.
	    $$
	    Therefore the theorem holds by the inductive assumption.

	    Similarly we can prove the theorem in the case where there exists
	    $(j,k) \in {\cal A}_2$ such that $z_{1jk} > 0$.
	    
	    Now we suppose that $z_{1jk} \ge 0$ for 
	    $(j,k) \in {\cal A}_1$ and $z_{1jk} \le 0$ for 
	    $(j,k) \in {\cal A}_1$.  
	    Since there does not exist the pattern in 
	    Figure \ref{fig:Case3}, there exist $k_4 < k_7 < k_3$ such that 
	    $z_{1jk} \ge 0$ for $k \le k_7$ and $z_{1jk} \le 0$ for $k >
	    k_7$.
	    This contradicts the condition $\sum_{k=1}^K k z_{1jk} = 0$
	    and hence $z$ is not a move. 
\begin{figure}
 \centering
 \begin{tabular}{cc}
  \includegraphics[scale=0.6]{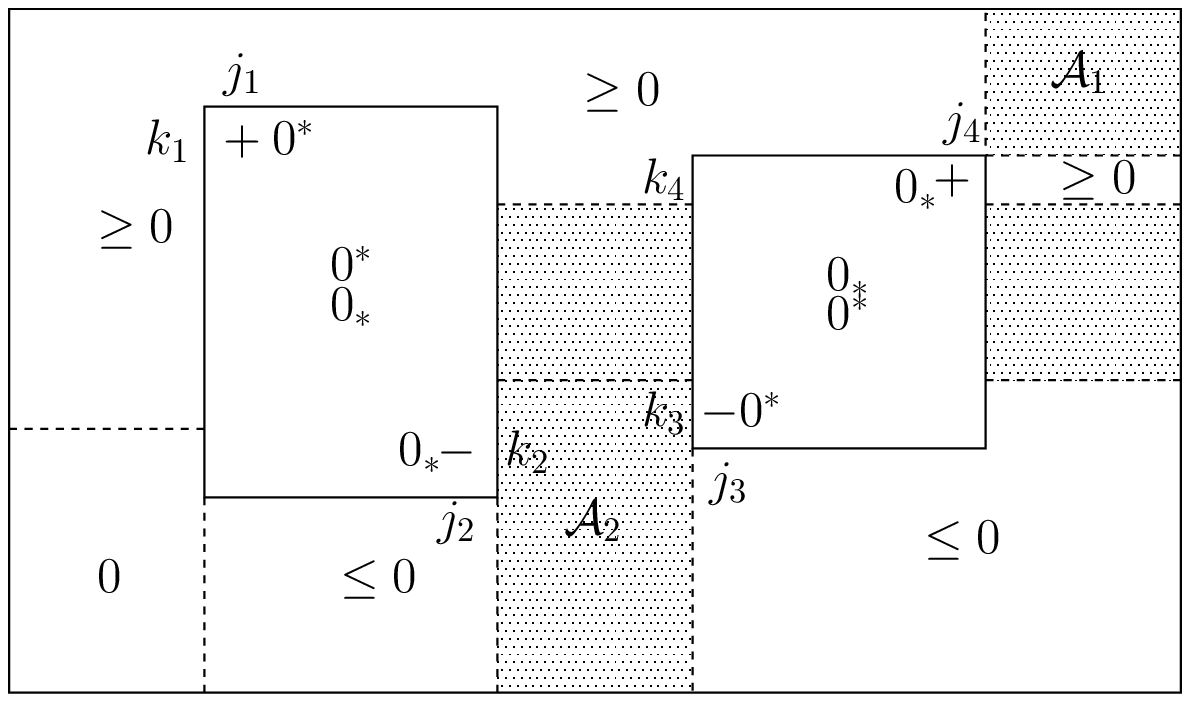} & 
  \includegraphics[scale=0.6]{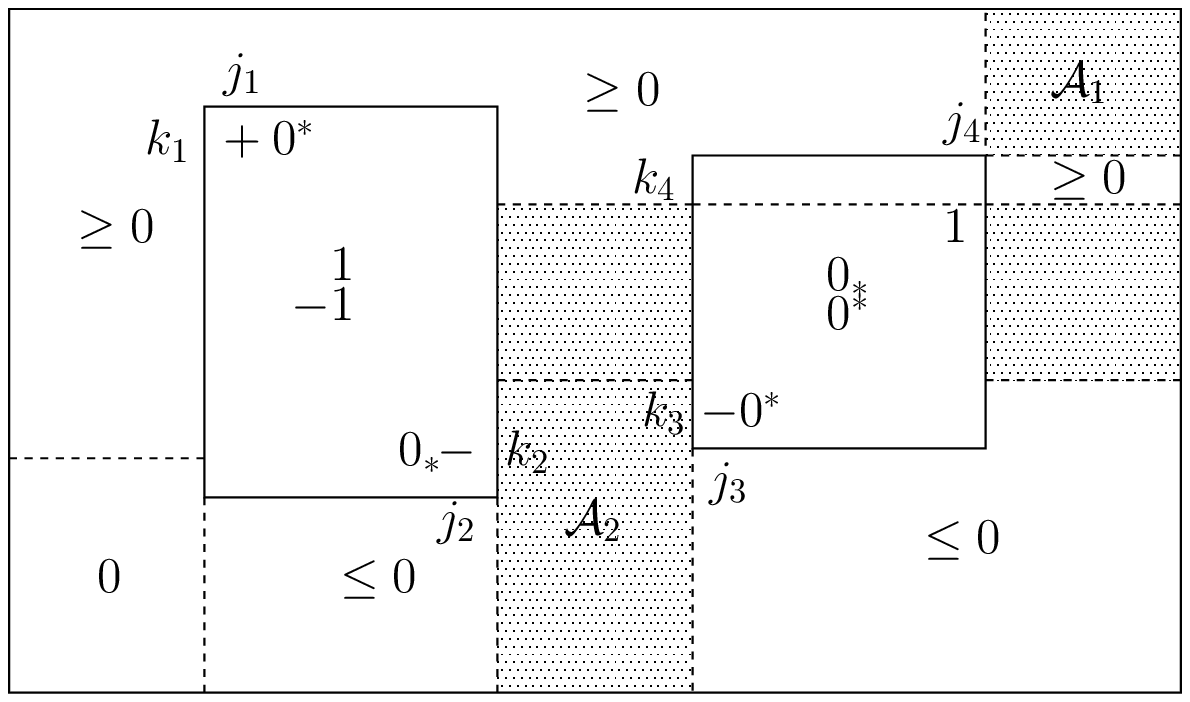}\\
  (i) & (ii)
 \end{tabular}
 \caption{Case 5-1}
 \label{fig:Case5-1}
\end{figure}

 \item[Case 5-2. ] $k_2 < k_4$\\
	    By using the same argument, we only need to consider the
	    case that $z^1$ contains patterns in Figure \ref{fig:Case5-2}(i).
	    Then 
	    both $S^+$ and $S^-$ is monotone in the sense of 
	    conditions 1 and 2 of Lemma \ref{ref:separation}. 
	    Therefore if $z$ is a move, we can reduce $L_1$ norm of $z$ 
	    from Lemma \ref{ref:separation} in the similar way to Case 5-1. 
\begin{figure}
 \centering
 \begin{tabular}{cc}
  \includegraphics[scale=0.6]{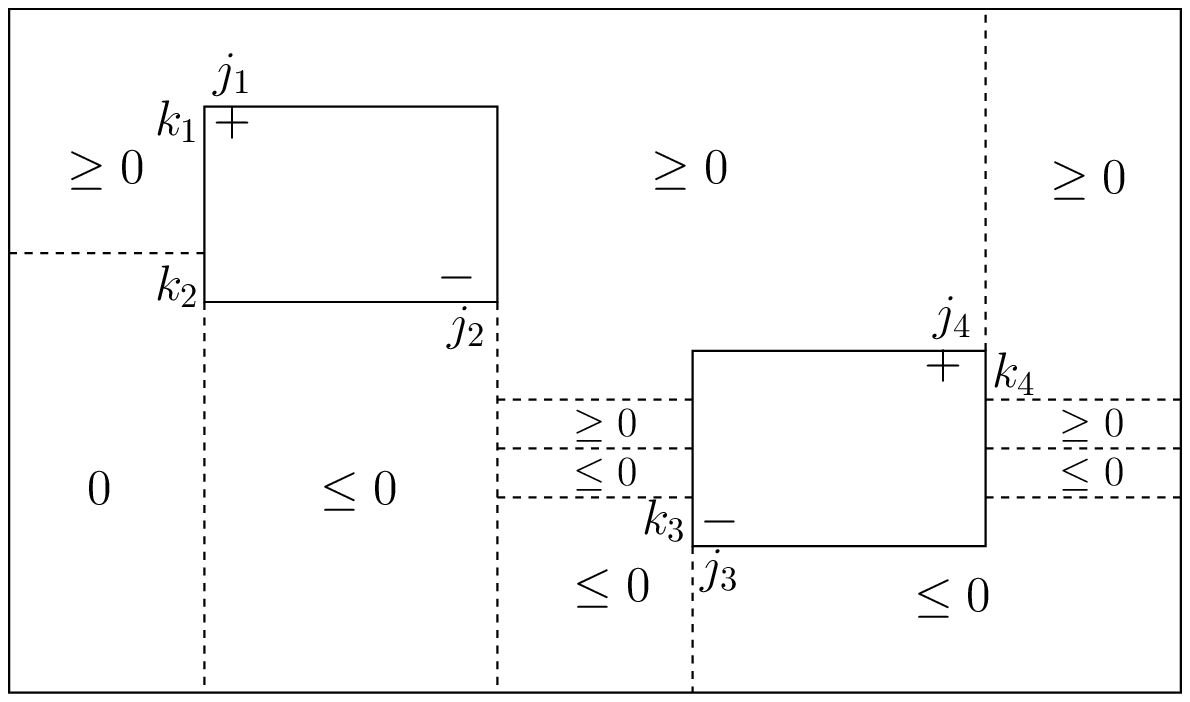} & 
  \includegraphics[scale=0.6]{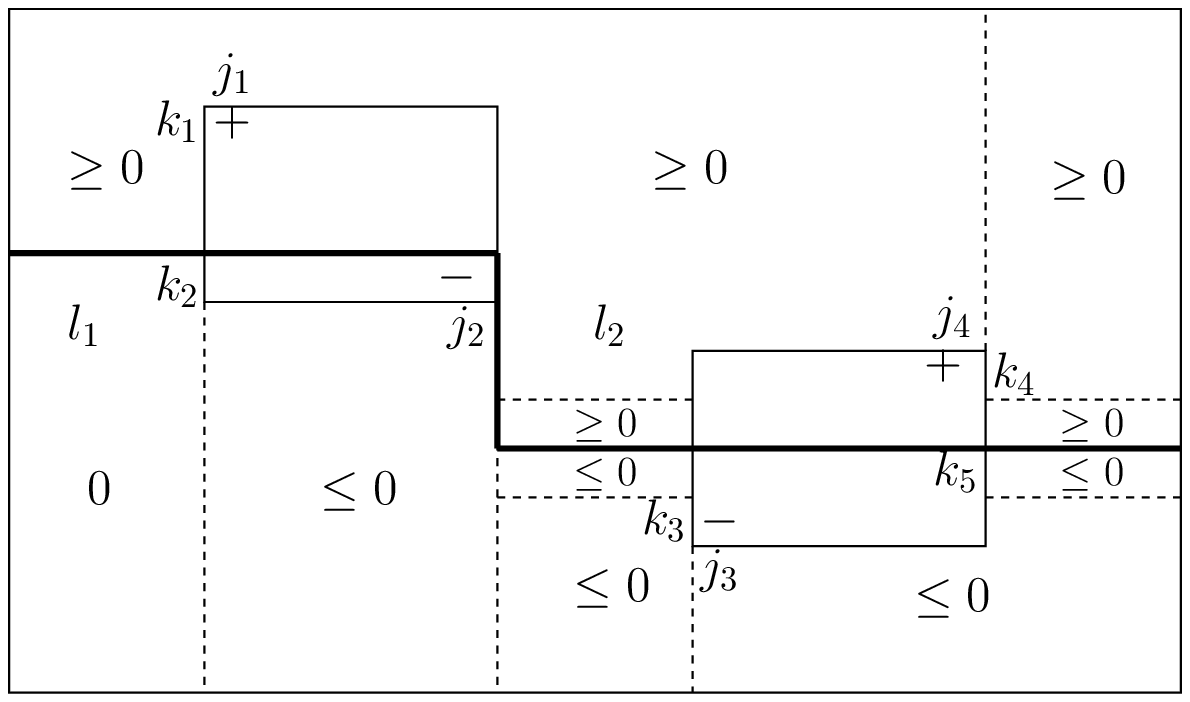}\\
  (i) & (ii)
 \end{tabular}
 \caption{Case 5-2}
 \label{fig:Case5-2}
\end{figure}
	    
\end{description}

\noindent {\bf Acknowledgments.}\ 
We are grateful to Kazuo Murota for valuable comments on the
separation lemma and to Ian Dinwoodie for constructive suggestion 
on the numerical examples.

\bibliographystyle{plainnat}
\bibliography{Hara-Takemura-Yoshida-arxiv}

\begin{thebibliography}{13}
\providecommand{\natexlab}[1]{#1}
\providecommand{\url}[1]{\texttt{#1}}
\expandafter\ifx\csname urlstyle\endcsname\relax
  \providecommand{\doi}[1]{doi: #1}\else
  \providecommand{\doi}{doi: \begingroup \urlstyle{rm}\Url}\fi

\bibitem[Agresti(1990)]{agre:1990}
Alan Agresti.
\newblock \emph{Categorical Data Analysis}.
\newblock New York : John Wiley and Sons, 1st edition, 1990.

\bibitem[Aoki and Takemura(2005)]{aoki-takemura-2005jscs}
Satoshi Aoki and Akimichi Takemura.
\newblock Markov chain {M}onte {C}arlo exact tests for incomplete two-way
  contingency table.
\newblock \emph{Journal of Statistical Computation and Simulation}, 75\penalty0
  (10):\penalty0 787--812, 2005.

\bibitem[Aoki et~al.(2008)Aoki, Hibi, Ohsugi, and Takemura]{ahot2008}
Satoshi Aoki, Takayuki Hibi, Hidefumi Ohsugi, and Akimichi Takemura.
\newblock {G}r\"obner bases of nested configurations.
\newblock \emph{Journal of Algebra}, 320:\penalty0 2583--2593, 2008.

\bibitem[Breslow and Day(1980)]{breslow-day1980}
N.E. Breslow and N.E. Day.
\newblock \emph{Statistical Methods in Cancer Research ({V}ol. 1): {T}he
  Analysis of Case-control Studies}.
\newblock Lyon : International Agency for Research on Cancer, 1980.

\bibitem[Chen et~al.(2008)Chen, Dinwoodie, and Yoshida]{Ian2008}
Y.~Chen, I.~Dinwoodie, and R.~Yoshida.
\newblock Markov chains, quotient ideals, and connectivity with positive
  margins.
\newblock \emph{{\em Algebraic and Geometric Methods in Statistics} dedicated
  to Professor Giovanni Pistone (P. Gibilisco, E. Riccomagno, M.-P. Rogantin,
  H. P. Wynn, eds.)}, 2008.
\newblock To appear.

\bibitem[Chen et~al.(2005)Chen, Dinwoodie, Dobra, and
  Huber]{chen-dinwoodie-dobra-huber2005}
Yuguo Chen, Ian Dinwoodie, Adrian Dobra, and Mark Huber.
\newblock Lattice points, contingency tables, and sampling.
\newblock In \emph{Integer points in polyhedra---geometry, number theory,
  algebra, optimization}, volume 374 of \emph{Contemp. Math.}, pages 65--78.
  Amer. Math. Soc., Providence, RI, 2005.

\bibitem[Christensen(1997)]{christensen1997}
Ronald Christensen.
\newblock \emph{Log-Linear Models and Logistic Regression}.
\newblock New York : Springer Verlag, 2nd edition, 1997.

\bibitem[Cornfield(1962)]{cornfield1962}
J.~Cornfield.
\newblock Joint dependence of risk of coronary heart disease on serum
  cholesterol and systolic blood pressure: A discriminant function analysis.
\newblock \emph{Fed. Proc.}, 21\penalty0 (11):\penalty0 58--61, 1962.

\bibitem[Diaconis and Sturmfels(1998)]{diaconis-sturmfels}
Persi Diaconis and Bernd Sturmfels.
\newblock Algebraic algorithms for sampling from conditional distributions.
\newblock \emph{Ann. Statist.}, 26\penalty0 (1):\penalty0 363--397, 1998.
\newblock ISSN 0090-5364.

\bibitem[Diaconis et~al.(1998)Diaconis, Eisenbud, and
  Sturmfels]{diaconis-eisenbud-sturmfels1998}
Persi Diaconis, David Eisenbud, and Bernd Sturmfels.
\newblock Lattice walks and primary decomposition.
\newblock In \emph{Mathematical essays in honor of Gian-Carlo Rota (Cambridge,
  MA, 1996)}, volume 161 of \emph{Progr. Math.}, pages 173--193. Birkh\"auser
  Boston, Boston, MA, 1998.

\bibitem[Moriguti(1953)]{moriguti1953}
Sigeiti Moriguti.
\newblock A modification of {S}chwarz's inequality with applications to
  distributions.
\newblock \emph{The Annals of Mathematical Statistics}, 24\penalty0
  (1):\penalty0 107--113, 1953.

\bibitem[Sturmfels(1996)]{sturmfels1996}
Bernd Sturmfels.
\newblock \emph{Gr\"obner Bases and Convex Polytopes}, volume~8 of
  \emph{University Lecture Series}.
\newblock American Mathematical Society, Providence, RI, 1996.
\newblock ISBN 0-8218-0487-1.

\bibitem[Takemura and Aoki(2005)]{takemura-aoki-2005bernoulli}
Akimichi Takemura and Satoshi Aoki.
\newblock Distance reducing {M}arkov bases for sampling from a discrete sample
  space.
\newblock \emph{Bernoulli}, 11\penalty0 (5):\penalty0 793--813, 2005.

\end{thebibliography}
\end{document}